\documentclass{amsart}

\usepackage[english]{babel}

\usepackage{amsmath}
\usepackage{graphicx}
\usepackage[colorlinks=true, allcolors=blue]{hyperref}
\usepackage{amsthm}
\usepackage{amssymb}
\usepackage{amsfonts,amscd}
\usepackage{fullpage}
\usepackage{subfig}

\usepackage{tikz}
\usetikzlibrary{positioning,automata,arrows,shapes,matrix, arrows, decorations.pathmorphing, calc}
\tikzset{main node/.style={circle,draw,minimum size=0.3em,inner
sep=0.5pt}}
\tikzset{state node/.style={circle,draw,minimum size=2em,fill=blue!20,inner
sep=0pt}}
\tikzset{small node/.style={circle,draw,minimum size=0.5em,inner
sep=2pt,font=\sffamily\bfseries}}
\usepackage{xspace}
\usepackage{adjustbox}
\theoremstyle{plain}
\newtheorem{theorem}{Theorem}[section]
\newtheorem{lemma}[theorem]{Lemma}
\newtheorem{prop}[theorem]{Proposition}

\theoremstyle{definition}
\newtheorem{defn}[theorem]{Definition}
\newtheorem{exa}[theorem]{Example}

\theoremstyle{remark}
\newtheorem{rmk}[theorem]{Remark}
\numberwithin{equation}{section}
\newcommand{\al}{\alpha}
\newcommand{\be}{\beta}
\newcommand{\ep}{\varepsilon}
\newcommand{\Q}{{\mathbb Q}}
\newcommand{\R}{{\mathbb R}}
\newcommand{\Z}{{\mathbb Z}}

\newcommand{\F}{{\mathbb F}}
\newcommand{\ra}{\rightarrow}

\newcommand{\bfn}{\mathbf{n}}
\newcommand{\bfw}{\mathbf{w}}

\newcommand{\Stab}{\text{\rm Stab}}

\newcommand{\simproots}{\Pi}

\newcommand{\allroots}{\Phi}

\newcommand{\Span}{\text{\rm Span}}

\newcommand{\rootht}{\text{\rm ht}}

\newcommand{\calf}{\mathcal{F}}

\newcommand{\calr}{\mathcal{R}}

\newcommand{\bfu}{\mathbf{u}}
\newcommand{\bfx}{\mathbf{x}}
\newcommand{\bfy}{\mathbf{y}}
\newcommand{\bfz}{\mathbf{z}}
\newcommand{\six}{{\mathbf 6}}
\newcommand{\seven}{{\mathbf 7}}
\newcommand{\eight}{{\mathbf 8}}

\newcommand{\wt}{\text{\rm wt}}

\newcommand{\Alt}{\text{\rm Alt}}

\newcommand{\qpx}{\Omega_{240}}
\newcommand{\qp}{\Omega}

\newcommand{\leh}{\kappa}
\newcommand{\lamp}{\Lambda^+}
\newcommand{\slm}{\lambda^-}
\newcommand{\slp}{\lambda^+}
\newcommand{\lamm}{\Lambda^-}

\newcommand{\vsix}{\Lambda}

\makeatletter
\newcommand{\exterior}[1]{\mathop{\mathpalette\exterior@{#1}}}
\newcommand{\exterior@}[2]{%
  \raisebox{\depth}{%
  \fontsize{\sf@size}{0}%
  \m@th
  $\ifx#1\displaystyle\textstyle\else#1\fi\bigwedge$}%
  ^{\mspace{-2mu}#2}%
  \kern-\scriptspace
}
\makeatother

\title{A partial order on the 240 packings of PG(3,2)}
\author{R.M. Green}
\address{Department of Mathematics, University of Colorado Boulder, Campus Box 395, Boulder, Colorado 80309-0395, USA}
\email{rmg@colorado.edu}
\date{\today}

\begin{document}

\begin{abstract}
It has long been known that the most symmetrical solutions of Kirkman's Schoolgirl Problem can be constructed
from the $240$ packings of the projective space $PG(3, 2)$, but it seems to have escaped notice that these packings 
have the structure of a partially ordered set. In this paper, we construct a shellable Bruhat-like graded partial order 
on the packings of $PG(3, 2)$ that refines the partial order on the product of four chains $[8]\times[5]\times[3]\times[2]$ 
and defines a Lehmer code on the packings. The partial order exists because the packings of $PG(3, 2)$ form a 
quasiparabolic set (in the sense of Rains--Vazirani) that is in bijective correspondence with a certain collection 
of maximal orthogonal subsets of the $E_8$ root system. The $E_8$ construction also induces transitive actions 
of the Weyl groups of type $D_n$ on the packings for $5 \leq n \leq 8$, and these actions are faithful for $n < 8$.
It is possible to define both the signed permutation action and the partial order using the combinatorics of labelled 
Fano planes.
\end{abstract}
\maketitle

\section{Introduction}\label{sec:intro}

Let $\F_2^4$ be a four dimensional vector space over the field $\F_2$ with two elements. The points, lines, and planes
of the smallest three dimensional projective space, $PG(3, 2)$, are the nonzero vectors of the subspaces of $\F_2^4$ of 
dimensions $1$, $2$, and $3$, respectively. The space $PG(3, 2)$ has 15 points, 35 lines, and 15 planes;
each line contains three points, and each plane contains seven points and seven lines. A {\it spread} of $PG(3, 2)$ is a 
partition of the points into five lines, and a {\it packing} of $PG(3, 2)$ is a partition of the lines into seven spreads. 
It is a classical result that each of the packings of $PG(3, 2)$ produces a solution to Kirkman's Schoolgirl 
Problem, and that the solutions with the largest possible automorphism groups all arise in this way (Remark \ref{rmk:ksp}).
The packings of $PG(3, 2)$ have been thoroughly studied from the point of view of design theory 
\cite[\S1]{mathon97overlarge} \cite[\S2]{mathon97seven} and finite geometry \cite[\S17]{hirschfeld85}. In this paper,
we will show that the packings also have an interesting structure as a partially ordered set, with points of contact
with the $E_8$ root system, combinatorial topology, Hecke algebras, and game theory.

There are $56$ spreads of $PG(3, 2)$, indexed by the 3-element subsets of the set $\eight = \{1, 2, \ldots, 8\}$, and 
there are $240$ packings, indexed by the $\eight$-labelled Fano planes, meaning the equivalence classes of labellings of
the vertices of the Fano plane by distinct elements of $\eight$. The $240$ packings form two orbits of size $120$ under 
the action of the alternating group $A_8 \cong PGL(4,2)$, and they form a single orbit under the action of $S_8$ 
\cite[Theorem 17.5.6 (ii)]{hirschfeld85}. Polster (\cite[\S5.1.3]{polster98}, \cite{polster99}) provides detailed 
illustrations and descriptions of the spreads and packings of $PG(3, 2)$. 

The main purposes of this paper are (a) to introduce and study a canonical Bruhat-like partial order on the packings,
and (b) to describe the relationship between the spreads and packings of $PG(3, 2)$ and the $E_8$ root system and 
thus to define explicit actions of Weyl groups of type $D$ on the packings. Each of these constructions can be
defined in an elementary way using the combinatorics of labelled Fano planes.

An important ingredient in these constructions is the multifaceted relationship between $\qpx$ and $\vsix = \F_2^6$
 which arises from the group isomorphism $A_8 \cong O_6^+(2)$. On the one 
hand, the set $\vsix$ is in bijection with the set $\lamp$ of all unordered partitions of $\eight$ into at 
most two sets of even size, which in turn is in bijection with the set of $64$ roots of type $E_8$ in which a certain 
simple root $\al_1$ appears with coefficient $1$. On the other hand, $\vsix$ is in bijection with $\lamm$, 
the set of partitions of $\eight$ into two subsets of odd size. Under this identification, $56$ of the elements of 
$\lamm$ correspond to the spreads of $PG(3, 2)$, and the other $8$ elements of $\lamm$ (which we call {\it basepoints}) 
correspond to the elements of $\eight$. The spreads (respectively, basepoints) are also in canonical bijection 
with the roots of type $E_8$ in which the simple root $\alpha_2$ appears with coefficient $1$ (respectively, $3$).

The $5$-dimensional projective space $PG(5, 2)$ arising from $\vsix\backslash \{0\}$ is related to $PG(3, 2)$ via 
the well-known {\it Klein correspondence}, which gives a bijection between the $35$ lines of $PG(3, 2)$ and the $35$ 
singular points of $PG(5, 2)$ with respect to a quadratic form $Q$ of Witt defect zero. There are $30$ projective 
planes in $PG(5, 2)$ that lie on the quadric, and we will prove in Theorem \ref{thm:corresp} that, for a particular 
choice of $Q$, the set of $240$ cosets $\qp^+_A$ of the corresponding $30$ totally singular solids (i.e., $3$-dimensional 
spaces) of $AG(6, 2)$ is in canonical bijection with the packings of $PG(3, 2)$. Under the identification between 
$\vsix$ and $\lamm$, each coset consists of seven spreads and one basepoint, where the seven spreads form a packing, 
and the basepoint corresponds to the unused label of a labelled Fano plane.

The identification of $\lamp$ with $\lamm$ induces a canonical bijection between $\qpx$ and the set 
$\qp_\Psi$ of $240$ maximal orthogonal subsets of $\Psi$. As described in \cite[\S6.3]{gx5},
the set $\qp_\Psi$ is a transitive quasiparabolic set in the sense of Rains--Vazirani \cite{rains13}, and this endows
$\qp_\Psi$ with a Bruhat-like graded partial order having a maximum and a minimum element. The partial order refines 
the order on the product of chains $[8]\times[5]\times[3]\times[2]$ (Theorem \ref{thm:refines}). This equips $\qpx$ with a
{\it Lehmer code} which shows that, despite their intricate structure, the packings of $PG(3,2)$ may be canonically 
labelled by a sequence of four small integers. The partial order can be defined concisely in terms of its rank function,
and this function may be defined in several equivalent ways, one of which uses the combinatorics of labelled Fano planes 
(Theorem \ref{thm:qppack}, Proposition \ref{prop:nonrec}), and two of which use the combinatorics of the $E_8$ root 
system (Remark \ref{rmk:residues}).

The set $\qp_\Psi$ supports transitive actions of the Weyl groups of types $D_5$ and $D_8$, as well as 
all the intermediate standard parabolic subgroups.
These induce canonically defined transitive actions of various groups of signed permutations on the set of
packings, where the sign changes act as translations from the $AG(6, 2)$ point of view. These actions can also be
defined solely in terms of the labelled Fano planes (Proposition \ref{prop:signact}). In the context of type $D_7$, 
we explain how the elements of $\qpx$ may be naturally indexed by $\seven$-labelled Fano planes whose edges are 
designated as positive or negative in such a way that each vertex is incident to evenly many negative edges. This 
notation makes it easier to prove (Theorem \ref{thm:faithful}) that the kernel of the action of $W(D_8)$ on the 
packings has order $2$, and that the other actions mentioned are faithful. In the context of types $D_6$ and 
$D_6 + A_1$, the elements of $\qpx$ are naturally indexed by $\six$-labelled {\it Pasch configurations} whose 
edges are designated as positive or negative in such a way that there are evenly many negative edges.

The paper is organized as follows. Section \ref{sec:prq} and Theorem \ref{thm:corresp} set up the bridges between 
three ways to parametrize the $240$ packings of $PG(3, 2)$: as translates of maximal totally singular subspaces 
of $\F_2^6$; as $\eight$-labelled Fano planes; and as certain maximal orthogonal sets of roots in the $E_8$ lattice. 
Section \ref{sec:lehmer} defines the main object of study, namely the Bruhat order on packings, and proves
(Theorem \ref{thm:refines}) that the order is a refinement of the partial order on a product of four chains
of sizes $2$, $3$, $5$, and $8$.

Section \ref{sec:signed} explains how the Weyl groups of type $D$, which are groups of signed permutations, 
act on the packings. Theorem \ref{thm:qppack} proves that the packings form a quasiparabolic set in the sense of
Rains--Vazirani, and that the Bruhat order on packings agrees with the partial order imposed by the
quasiparabolic structure.

Section \ref{sec:d7} concentrates on the action of the Weyl group of type $D_7$ on the packings and explains
how to parametrize the packings by a set of $\seven$-labelled Fano planes (or a set of $\six$-labelled Pasch
configurations) with signed edges. The main result of Section \ref{sec:d7} is Theorem \ref{thm:faithful}, 
which proves that $W(D_7)$ acts faithfully on the packings, and that $W(D_8)$ acts faithfully modulo its centre.

In Section \ref{sec:conc}, we give two combinatorial interpretations of the rank function of the Bruhat order
on packings in terms of the $E_8$ root system, and we outline how the partial order relates to combinatorial 
topology, Hecke algebras, and game theory. We also discuss the significance of the number of packings of 
$PG(3, 2)$ being the same as the number of roots of type $E_8$.

\section{Packings, root systems, and a quadratic form}\label{sec:prq}

Section \ref{sec:prq} introduces the basic definitions we need in the sequel. The main result is Theorem 
\ref{thm:corresp}, which explains how the packings of $PG(3, 2)$ may be identified in a canonical way with
certain orthogonal subsets in the $E_8$ root system, as well as with the cosets of the maximal totally singular
subspaces of a certain $6$-dimensional vector space over $\F_2$.

We write $\bfn$ for the set $\{1, 2, \ldots, n\}$, and $\binom{\bfn}{k}$ for the set of all size-$k$ subsets of 
$\bfn$. 

Let $\F_2^8 = \Span(\ep_1, \ep_2, \ldots, \ep_8)$ be an $8$-dimensional vector space over the field $\F_2$ with two 
elements, and let $V$ be the $7$-dimensional $\F_2$-vector space $\F_2^8/U$, where $U$ is the span of the all-ones 
vector. The vectors of $V$ are parametrized by the unordered partitions of $\eight$ into at most two parts, where
we identify the unordered partition $\{A, B\}$ of $\eight$ with the coset of $U$ represented by $$
v_A = \sum_{i \in A} \ep_i \text{\quad or \quad} v_B = \sum_{j \in B} \ep_j
.$$ If $A = \{a_1, a_2, \ldots, a_k\}$ then we write $v_{a_1, a_2, \ldots, a_k}$ for $v_A$. For $0 \leq i \leq 4$, 
we write $P_i$ for the set of unordered partitions of $\eight$ into two parts of size $i$ and $8-i$, and we write
$V_i$ for the subset of $V$ corresponding to $P_i$, so that we have $V = \dot\bigcup_{i = 0}^4 V_i$. Note that for 
$i = 0$, $1$, $2$, $3$, and $4$, we have $|V_i| = 1$, $8$, $28$, $56$, and $35$, respectively. The vector space 
$V = \lamp \oplus \lamm$ has a $\Z/2\Z$-grading induced by taking Hamming weight modulo $2$, where the $6$-dimensional 
subspaces $\lamp$ and $\lamm$ are defined by $\lamp = V_0 \cup V_2 \cup V_4$ and $\lamm = V_1 \cup V_3$. Every vector 
$v \in \lamm$ is uniquely of the form $v_i$ or $v_{ijk}$. For $1 \leq i \leq 8$, we define the operators 
$t_i : V \ra V$ by $t_i(v + U) = v + \ep_i + U$. It follows from the definitions that $t_i(\lamp) = \lamm$ and 
$t_i(\lamm) = \lamp$.

The well-known combinatorial model of $PG(3, 2)$ described in \cite[\S2]{hall80} and \cite[Appendix 1B]{wagner61}
gives a bijection $\tau$ from the set of $35$ lines of $PG(3, 2)$ to the set $\binom{\seven}{3}$. If $p$ is a point or a 
plane of $PG(3, 2)$, then $p$ is incident with seven lines, $l_1, l_2, \ldots, l_7$, and we identify $p$ with the subset 
$\tau(p) = \{\tau(l_1), \ldots, \tau(l_7)\}$ of $\binom{\seven}{3}$.

If $x$ and $x'$ are distinct points, then $\tau(x) \cap \tau(x') \subset \binom{\seven}{3}$ is the singleton corresponding 
to the unique line containing both $x$ and $x'$. Dually, if $y$ and $y'$ are distinct planes, then 
$\tau(y) \cap \tau(y') \subset \binom{\seven}{3}$ is the singleton corresponding to the unique line contained in both 
$y$ and $y'$.

Each point lies in seven planes, and each plane contains seven points. If the point $x$ does not lie in the plane $y$,
then $\tau(x)$ is disjoint from $\tau(y)$, whereas if $x$ lies in $y$, then $\tau(x) \cap \tau(y)$ is the $3$-element subset
of $\binom{\seven}{3}$ that corresponds to the three lines in the plane $y$ that contain the point $x$.

If $l$ and $l'$ are distinct lines, then the following are equivalent:
\begin{itemize}
\item[(i)]{$l$ and $l'$ are incident with a common point;}
\item[(ii)]{$l$ and $l'$ are incident with a common plane;}
\item[(iii)]{the set $l \cap l' \subset \seven$ is a singleton.}
\end{itemize}

\begin{exa}\label{exa:septuples}
Consider the $7$-tuples of lines $$
A = \{124, 136, 157, 235, 267, 347, 456\} \quad \text{and} \quad B = \{127, 136, 145, 235, 246, 347, 567\}
.$$ Each pair of triples in $A$ intersects in a singleton, and the same applies to the triples in $B$. This implies that 
the triples appearing in $A$ (or in $B$) either represent the seven lines incident with a given point, or the seven lines 
incident with a given plane. Since $A \cap B = \{136, 235, 347\}$ is nonempty and not a singleton, we conclude that $A$ 
represents a point that lies in the plane $B$, or vice versa, and that there are precisely three lines incident with both 
$A$ and $B$, namely $136$, $235$, and $347$.
\end{exa}

There is a bijection $\iota : P_4 \rightarrow \binom{\seven}{3}$ given by 
$\iota(v_{ijk8}) = \{i, j, k\}$. For example, we have $$
\iota(\{\{1, 4, 6, 7\}, \{2, 3, 5, 8\}\}) = \iota(v_{2358}) = \{2, 3, 5\}
,$$ and $$
\iota(\{\{1, 4, 6, 8\}, \{2, 3, 5, 7\}\}) = \iota(v_{1468}) = \{1, 4, 6\}
.$$

A {\it spread} of $PG(3, 2)$ is a partition of the $15$ points into five lines. The $56$ spreads of $PG(3, 2)$ are parametrized 
by the set $\binom{\eight}{3}$, where the spread indexed by the set $A = \{i, j, k\} \subset \eight$ corresponds to
the $5$-element subset $\sigma(A) \subset \binom{\seven}{3}$ given by $$
\sigma(A) = \left\{\iota(v_T) : A \subset T \in \binom{\eight}{4}\right\}
.$$ Such a set $\sigma(A)$ is a spread because any two distinct elements $z, z' \in \sigma(A)$ intersect in $0$ or $2$ 
elements, which implies that $z$ and $z'$ correspond to non-intersecting lines. In particular, if $8 \not\in A$ then 
$\sigma(A)$ consists of $A$ itself together with the four $3$-element subsets of $\seven \backslash A$. If, on the other 
hand, we have $A = \{i, j, 8\}$, then $\sigma(A)$ consists of the five elements of $\binom{\seven}{3}$ that contain 
both $i$ and $j$. 

\begin{exa}\label{exa:spreads}
The five lines in the spread indexed by the triple $368$ are $$
\{136, 236, 346, 356, 367\}
,$$ and the five lines in the spread indexed by the triple $257$ are $$
\{134, 136, 146, 257, 346\}
.$$
\end{exa}

A {\it packing} of $PG(3, 2)$ is a partition of the $35$ lines into seven spreads. The spreads indexed by the sets 
$A, B \subset \eight$ consist of disjoint sets of lines if and only if $A \cap B \subset \eight$ is a singleton, so a 
packing is indexed by a collection of seven elements of $\binom{\eight}{3}$, any two of which intersect in a singleton. 
The data that defines a packing may be concisely represented as a labelling of vertices of the Fano plane $PG(2, 2)$ 
by elements of $\eight$, up to automorphisms of $PG(2, 2)$. The automorphism group of $PG(2, 2)$ is $PGL(3, 2)$, which 
has order $168$, thus giving rise to $8!/168=240$ distinct packings. It is known that all spreads and all packings 
of $PG(3, 2)$ arise from the constructions just described (see \cite[4.2]{hall80} and \cite{magaud22}).

\begin{defn}\label{def:slm}
Given an $\eight$-labelled Fano plane $x$, we define $\slm(x) \in \lamm$ to be the set consisting of
the seven elements $v_{ijk}$ for each spread $ijk$ in $x$ together with the basepoint $v_l$, where
$l$ is the unused label in the Fano plane. 
\end{defn}

The $\eight$-labelled Fano planes $x_0$ and $x_1$ are depicted in Figure \ref{fig:fano1}. It will turn out that 
$x_0$ (respectively, $x_1$) indexes the minimum (respectively, maximum) element in the partial order on the packings 
of $PG(3, 2)$. In the notation of Definition \ref{def:slm}, we have $$
\slm(x_0) = \{ v_{127},
v_{136},
v_{145},
v_{235},
v_{246},
v_{347},
v_{567},
v_{8}\}
$$ and $$
\slm(x_1) = \{v_{234},
v_{256},
v_{278},
v_{357},
v_{368},
v_{458},
v_{467},
v_{1}\}
.$$

\begin{figure}
\begin{center}
\subfloat[$x_0$]{
\begin{tikzpicture}[
mydot/.style={
  draw,
  circle,
  fill=black,
  inner sep=1.5pt}
]
\draw
  (0,0) coordinate (A) --
  (3,0) coordinate (B) --
  ($ (A)!.5!(B) ! {sin(60)*2} ! 90:(B) $) coordinate (C) -- cycle;
\coordinate (O) at
  (barycentric cs:A=1,B=1,C=1);
\draw (O) circle [radius=3*1.717/6];
\draw (C) -- ($ (A)!.5!(B) $) coordinate (LC); 
\draw (A) -- ($ (B)!.5!(C) $) coordinate (LA); 
\draw (B) -- ($ (C)!.5!(A) $) coordinate (LB); 
\foreach \Nodo in {A,B,C,O,LC,LA,LB}
  \node[mydot] at (\Nodo) {};    
  \node [left=0.1cm of A] {$5$};
  \node [right=0.1cm of B] {$6$};
  \node [above=0.1cm of C] {$2$};
  \node [right=0.1cm of LA] {$4$};
  \node [left=0.1cm of LB] {$3$};
  \node [below=0.1cm of LC] {$7$};
  \node [above right=0.1cm and 0.001cm of O] {$1$};
\end{tikzpicture}}\quad\quad\quad\quad
\subfloat[$x_1$]{
\begin{tikzpicture}[
mydot/.style={
  draw,
  circle,
  fill=black,
  inner sep=1.5pt}
]
\draw
  (0,0) coordinate (A) --
  (3,0) coordinate (B) --
  ($ (A)!.5!(B) ! {sin(60)*2} ! 90:(B) $) coordinate (C) -- cycle;
\coordinate (O) at
  (barycentric cs:A=1,B=1,C=1);
\draw (O) circle [radius=3*1.717/6];
\draw (C) -- ($ (A)!.5!(B) $) coordinate (LC); 
\draw (A) -- ($ (B)!.5!(C) $) coordinate (LA); 
\draw (B) -- ($ (C)!.5!(A) $) coordinate (LB); 
\foreach \Nodo in {A,B,C,O,LC,LA,LB}
  \node[mydot] at (\Nodo) {};    
  \node [left=0.1cm of A] {$7$};
  \node [right=0.1cm of B] {$6$};
  \node [above=0.1cm of C] {$3$};
  \node [right=0.1cm of LA] {$8$};
  \node [left=0.1cm of LB] {$5$};
  \node [below=0.1cm of LC] {$4$};
  \node [above right=0.1cm and 0.001cm of O] {$2$};
\end{tikzpicture}}
       \caption{The inequivalent labellings of the Fano plane corresponding to $x_0$ and $x_1$}
\label{fig:fano1}
\end{center}
\end{figure}
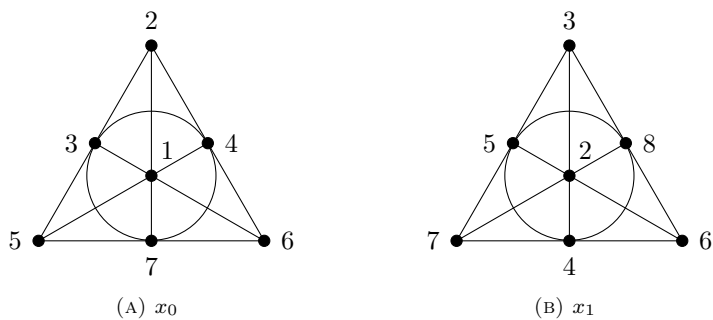

\begin{rmk}\label{rmk:ksp}
In {\it Kirkman's Schoolgirl Problem}, proposed by T.P. Kirkman in 1850, fifteen girls go on a walk in five groups 
of three on seven consecutive days. The problem is to schedule the walks in such a way that each pair of girls walks 
together precisely once. 

It is a classical result \cite{conwell1910} that each packing of $PG(3, 2)$ provides a solution to the problem,
where the $15$ points are the girls, the seven spreads are the days, and the five lines in a spread create a partition 
of the $15$ girls into five sets of three. This solves the problem because any two points of $PG(3, 2)$ lie on a 
unique line, which in turn lies in a unique spread.

Up to permutations of the girls and the days, this method produces all the solutions to Kirkman's Schoolgirl Problem 
having the largest possible automorphism group, namely the group $PGL(3, 2)$ of order $168$ \cite{cole1922}. 
\end{rmk}

\begin{defn}\label{def:e8}
The {\it $E_8$ lattice} $\Gamma_8$ is the rank $8$ free abelian subgroup of $\R^8$ consisting of the points $$
\left\{ (x_i) \in \Z^8 \cup \left( \Z + \tfrac{1}{2} \right)^8 : \ \sum_i x_i = 0 \mod 2\right\}
.$$ The set of {\it roots}, $\allroots$, of type $E_8$ is the set of $240$ elements of $\Gamma_8$ of norm $2$.
Let $\{e_1, e_2, \ldots, e_8\}$ be an orthonormal basis for $\R^8$. There are $112$ roots of the form
$\pm e_i \pm e_j$ where $1 \leq i < j \leq 8$ and the signs are chosen independently. The other $128$ roots are
of the form $\sum_{i = 1}^8 \pm \tfrac{1}{2}e_i$, where the signs are chosen so that there are evenly many minus 
signs. 
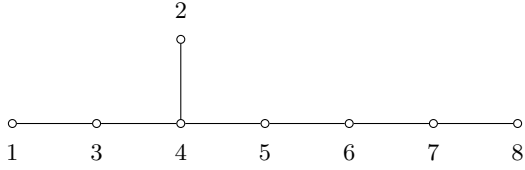
\begin{figure}
\begin{tikzpicture}
    \node[main node] (1) {};
            \node[main node] (3) [right=1cm of 1] {};
            \node[main node] (4) [right=1cm of 3] {};
            \node[main node] (5) [right=1cm of 4] {};
            \node[main node] (6) [right=1cm of 5] {};
            \node[main node] (7) [right=1cm of 6] {};
            \node[main node] (8) [right=1cm of 7] {};
            \node[main node] (2) [above=1cm of 4] {};

            \path[draw]
            (1)--(3)--(4)--(5)--(6)--(7)--(8)
            (2)--(4);

            \node (11) [below=0.1cm of 1] {\small{$1$}};
            \node (22) [above=0.1cm of 2] {\small{$2$}};
            \node (33) [below=0.1cm of 3] {\small{$3$}};
            \node (44) [below=0.1cm of 4] {\small{$4$}};
            \node (55) [below=0.1cm of 5] {\small{$5$}};
            \node (66) [below=0.1cm of 6] {\small{$6$}};
            \node (77) [below=0.1cm of 7] {\small{$7$}};
            \node (88) [below=0.1cm of 8] {\small{$8$}};

\end{tikzpicture}
\caption{Dynkin diagram of type $E_8$}
\label{fig:e8}
\end{figure}

Following \cite[\S2]{humphreys90}, 
we define the set of {\it simple roots}, $\simproots$, of type $E_8$ to be the vectors $\al_1, \ldots, \al_8$ by $$
\al_1 = \tfrac{1}{2}e_1 - \tfrac{1}{2}e_2 - \tfrac{1}{2}e_3 - \tfrac{1}{2}e_4 - \tfrac{1}{2}e_5 - \tfrac{1}{2}e_6 - 
\tfrac{1}{2}e_7 + \tfrac{1}{2}e_8 ,$$ $\al_2 = e_1 + e_2$, and $\al_i = e_{i-1} - e_{i-2}$ for all $3 \leq i \leq 8$.
Every root is an integer linear combination of simple roots with coefficients of like sign, and the root is called
{\it positive} (respectively, {\it negative}) if the coefficients are all nonnegative (respectively, nonpositive).
There is a partial order on roots in which $\al \leq \be$ if $\be - \al$ is a nonnegative linear combination of 
simple roots. The unique maximal element in this partial order is the highest root $\theta = e_7 + e_8$.
Each simple root $\al_i$ defines a {\it simple reflection} $s_{\al_i}$, which acts on the $E_8$-lattice by reflection in the
simple root $\al_i$. The {\it Weyl group} $W(E_8)$ is the group generated by the set $S$ of simple reflections.
We define $\allroots_{c, d} \subset \allroots$ to be the set of roots in which $\alpha_c$ appears with coefficient $d$.
\end{defn}

\begin{lemma}\label{lem:e8}
Suppose that $\al$ is a root such that $$
\al = \sum_i c_i \al_i = \sum_i d_i e_i
.$$ 
\begin{itemize}
\item[{\rm (i)}]{If $k$ is maximal such that $d_k$ is nonzero, then $\al$ is positive if and only if $d_k$ is positive.}
\item[{\rm (ii)}]{We have $2d_8 = c_1$.}
\item[{\rm (iii)}]{There are $14$ roots for which $c_1 = 2$, namely $\ep_8 \pm \ep_i$ for both choices of sign and all
$i$ such that $1 \leq i \leq 7$.}
\item[{\rm (iv)}]{There are $64$ roots for which $c_1 = 1$, namely those of the form $$
\tfrac{1}{2}e_8 + \sum_{i = 1}^7 \pm \tfrac{1}{2}e_i
,$$ where the signs are chosen so that there are evenly many minus signs.}
\end{itemize}
\end{lemma}

\begin{proof}
Statements (i) and (ii) can be checked by inspection for the simple roots, and the general case follows by linearity.

If $c_1 = 2$ (respectively, $c_1 = 1$), then we must have $d_8 = 1$ (respectively, $d_8 = \tfrac{1}{2}$) by (ii). 
Parts (iii) and (iv) now follow from the description of the roots given in Definition \ref{def:e8}.
\end{proof}

\begin{defn}\label{def:psi}
We define $\Psi := \allroots_{1, 1}$ to be the set of $64$ roots of type $E_8$ in which $\al_1$ appears with 
coefficient $1$. We define $\qp_\Psi$ to be the set of $8$-element subsets of $\Psi$ that are orthogonal with respect 
to the usual scalar product on $\R^8$.
Given $\al \in \Psi$, we define $\slp(\al) \in \lamp$ by $\slp(\al) = \{A \ |\  B \}$, where we have $8 \in A$ and $$
\al = \sum_{i \in A} \tfrac{1}{2}e_i - \sum_{j \in B} \tfrac{1}{2}e_j
.$$
\end{defn}

\begin{rmk}\label{rmk:slp}
It is immediate from Lemma \ref{lem:e8} (iv) that $\slp$ is a bijection from $\Psi$ to $\lamp$.
\end{rmk}

The {\it Hamming weight} of a binary string is the number of occurrences of $1$, and the {\it Hamming distance}
between two binary strings is the number of positions at which the strings disagree. The subset $E \subset \F_2^8$ 
consisting of vectors of even weight is a subspace of $\F_2^8$.
The quadratic form $$
Q'\left(\sum_{i = 1}^8 x_i \ep_i\right) = \sum_{1 \leq i < j \leq 8} x_i x_j
$$ on $E$ satisfies $Q'(v) = \binom{\wt(v)}{2} \mod 2$, where $\wt(v)$ is the Hamming weight of $v$. 
The form $Q'$ is constant on the cosets $E/U$, and it induces a quadratic form $Q$ on $\lamp$ such that $Q(v) = 1$ if 
$v \in V_2$, and $Q(v) = 0$ if $v \in V_0 \cup V_4$. In
other words, we have $Q(v) = \wt(v)/2 \mod 2$, where $\wt(v)$ denotes the Hamming weight of either vector in the coset
corresponding to $v$. The form $Q$ is well known in the literature in the context of the orthogonal group 
$O_6^+(2) \cong A_8$. In particular, the entry for the group $A_8$ in the Atlas of Finite Simple Groups \cite{conway85}
states that $Q$ has Witt defect zero (and thus Witt index $3$).

If $v \in V = \F_2^8/U$, then the coset $v + U$ lies in $V_0$ (respectively, $V_1$, $V_2$, $V_3$, 
$V_4$) if we have $\wt(v) \in \{0, 8\}$ (respectively, $\wt(v) \in \{1, 7\}$, $\wt(v) \in \{2, 6\}$, 
$\wt(v) \in \{3, 7\}$, $\wt(v) = 4$). It is therefore well defined to call two cosets $v_1 + U$ and $v_2 + U$ in $V$ 
{\it separated} if the Hamming distance between $v_1$ and $v_2$ is $4$. We call a subset $\calr$ of $V$ {\it maximally 
separated} if $\calr$ is a subset of maximal cardinality such that any two elements of $\calr$ are separated.

The next result relates the packings of $PG(3, 2)$ to the subset $\Psi$ of the $E_8$ root system and 
the quadratic form $Q$ on $\lamp$. 

\begin{theorem}\label{thm:corresp}
Let $Q$ be the quadratic form on $\lamp$ defined above, where $V = \F_2^8/U = \lamp \oplus \lamm$ as vector spaces.
Let $\qp_A$ be the set of maximally separated subsets of $V$, and define $\qp^+_A := \qp_A \cap \lamp$ and
$\qp^-_A := \qp_A \cap \lamm$.
\begin{itemize}
\item[{\rm (i)}]{The elements of $\qp_A$ are the cosets in $V$ of the $30$ maximal totally singular subspaces of 
$AG(6, 2) \cong \lamp$ with respect to $Q$. In particular, every element of $\qp_A$ is a set of size $8$.}
\item[{\rm (ii)}]{We have $|\qp_A|=480$, $|\qp^+_A| = |\qp^-_A| = 240$, and $\qp_A = \qp^+_A \ \dot\cup \ \qp^-_A$.}
\item[{\rm (iii)}]{The function $\slm$ is a bijection from the set $\qpx$ of packings of $PG(3, 2)$ to $\qp^-_A$.}
\item[{\rm (iv)}]{The function $\slp$ is a bijection from the set $\qp_\Psi$ of orthogonal $8$-tuples in $\Psi$ to
$\qp^+_A$.}
\item[{\rm (v)}]{The involution $t_8 : V \ra V$ given by $t_8(v + U) = v + \ep_8 + U$ induces a 
bijection $$
\slm \circ t_8 \circ (\slp)^{-1} : \qp_\Psi \ra \qpx
$$ from the set of $240$ orthogonal $8$-tuples in $\Psi$ to the set of packings of $PG(3, 2)$.}
\end{itemize}
\end{theorem}

\begin{proof}
Let $\Gamma$ be the graph with vertex set $P_4$ in which $x$ is adjacent to $y$ if and only if $x$ and $y$ are
separated. It is known \cite[\S10.13]{brouwer22} that $\Gamma$ is a strongly regular graph with parameters 
$(35, 18, 9, 9)$, and that $\Gamma$ is the complement of the unique rank $3$ strongly regular graph with parameters 
$(35, 16, 6, 8)$. Furthermore, the maximal cliques in $\Gamma$ have size $7$, and correspond (when the zero vector is 
included) to the $30$ maximal totally isotropic subspaces of $\lamp$, meaning the maximal subspaces on which $Q$
vanishes. These $30$ subspaces are maximally separated subsets of $\lamp \subset V$ because they have $8$ elements. The 
cosets of the $30$ subspaces in $V$ are also maximally separated, because $v_1 + U$ and $v_2 + U$ are separated if
and only if $v_1 + a + U$ and $v_2 + a + U$ are separated for any given $a$.

It remains to show that there are no additional maximally separated subsets of $V$.
Suppose that $A \subset V$ is a maximally separated subset of $V$, and let $a \in A$. This implies that the translate
$A - a$ is a maximally separated subset of $V$ that contains the zero vector, and that $(A - a) \backslash \{0\}$ is
a maximal clique in $\Gamma$. The classification of the cliques of $\Gamma$ given in the previous paragraph implies
that $A - a$ is a totally isotropic subspace, which completes the proof of (i).

Two binary strings at Hamming distance $4$ have the same Hamming weight modulo $2$, which implies that any two
separated elements of $V$ must both lie in $\lamp$ or in $\lamm$, and therefore that we have a disjoint union 
$\qp_A = \qp^+_A \ \dot\cup \ \qp^-_A$. The function $t_8 : \lamp \ra \lamm$ defined in the statement of (v)
induces a bijection between $\qp^+_A$ and $\qp^-_A$, which implies that $|\qp^+_A| = |\qp^-_A|$ and that
$|\qp_A| = |\qp^+_A| + |\qp^-_A|$. Recall that we have $\dim(V)=7$, so each totally isotropic subspace has 
$2^{7-3}=16$ cosets in $V$. It follows from (i) that $|\qp_A|=480$, and this completes the proof of (ii). 

Let $x \in \qpx$. Definition \ref{def:slm} implies that if $v$ and $v'$ are two elements of 
$\slm(x) \subset \lamm$, then either (a) we have $v = v_{ijk}$, $v' = v_{ilm}$ and $|\{i, j, k, l, m\}| = 5$, 
or (b) $v = v_{ijk}$, $v' = v_l$, and $|\{i, j, k, l\}| = 4$. In either case, $v$ and $v'$ are separated.
Definition \ref{def:slm} implies that $|\slm(x)| = 8$, from which it follows that $\slm(x)$ is maximally
separated. The function $\slm$ is injective, and we have $|\qpx| = 240$, so (ii) implies that $\slm$ is a
bijection, which completes the proof of (iii).

Two elements of $\Psi$ are orthogonal in $\R^8$ if and only if their coordinates in terms of the $e_i$-basis
agree in precisely four positions. This implies that $\slp$ induces a bijection between the mutually orthogonal 
size $8$ subsets of $\Psi$ and the maximally separated subsets of $\lamp$, proving (iv).

Part (v) is a consequence of parts (iii) and (iv).   
\end{proof}

The specific choice of bijection $t_8$ in part (v) of Theorem \ref{thm:corresp} is made for compatibility with 
Lemma \ref{lem:d8act}.

The following equivalence relation on packings is already implicit in the literature.

\begin{defn}\label{def:coseteq}
We call two elements of $\qp^+_A$ {\it coset equivalent} if they are cosets of the same subspace. By extension,
this induces a notion of coset equivalence on $\qpx$ (respectively, $\qp_\Psi$) via $\slm$ (respectively, $\slp$).
\end{defn}

\begin{exa}\label{exa:thcthn}
The motivation for studying $\qp_\Psi$ is that it has the structure of a partially ordered set, as described in
\cite[\S6.3]{gx5}. If we denote the minimum and maximum elements of $\qp_\Psi$ by $y_0$ and $y_1$ respectively,
then it follows from \cite[Remark 6.10 (ii)]{gx5} that we have $$
\slp(y_0) = \{v_{1278}, v_{1368}, v_{1458}, v_{2358}, v_{2468}, v_{3478}, v_{5678}, v_\emptyset\}
$$ and $$
\slp(y_1) = \{v_{18}, v_{2348}, v_{2568}, v_{3578}, v_{4678}, v_{45}, v_{36}, v_{27}\}
.$$ A routine verification shows that we have $\slp(y_0) = t_8 \slm(x_0)$ and $\slp(y_1) = t_8 \slm(x_1)$,
which is the motivation for the definition of $x_0$ and $x_1$. The vectors $v_0, v_1 \in \qp^+_A$ that correspond 
to $y_0$ and $y_1$ satisfy $v_0 + U = v_1 + \ep_1 + \ep_8 + U$, which proves that the minimum and maximum elements 
are coset equivalent.
\end{exa}

\begin{rmk}\label{rmk:hhz}
The $6$-dimensional vector space underlying $AG(6, 2)$ may be identified with $\Alt(4, 2)$, the set of all
skew-symmetric $4 \times 4$ matrices over $\F_2$, i.e., the symmetric matrices with zeros
on the diagonal. The determinant of a skew symmetric matrix $U$ is the square of the Pfaffian, so that we have $$
\det(U) = (u_{12}u_{34} - u_{13}u_{24} + u_{14}u_{23})^2
,$$ where $u_{ij}$ is the $(i, j)$-entry of $U$. Any skew symmetric matrix has even rank, even in characteristic $2$,
so any nonzero singular matrix in $\Alt(4, 2)$ has rank $2$ and is a zero of the quadratic form 
$u_{12}u_{34} + u_{13}u_{24} + u_{14}u_{23}$ of Witt index 3.

From this point of view, the $35$ singular points of $PG(5, 2)$  correspond to the $35$ nonzero singular matrices in
$\Alt(4, 2)$, and the elements of $\qp^+_A$ correspond to the $30$ $3$-dimensional subspaces of $\Alt(4, 2)$ that 
consist entirely of singular matrices, together with all of their translates.

The elements of $\Alt(4, 2)$ can be made into a graph, where two vertices are adjacent if their difference has rank
$2$. For more details, see \cite{huang15} and \cite[\S3.4.2]{brouwer22}.
\end{rmk}

\begin{rmk}\label{rmk:otherpsi}
The spreads and packings of $PG(3, 2)$ may also be modelled by the set of $64$ roots 
$\Psi' := \allroots_{2, 1} \ \dot\cup \ \allroots_{2, 3}$ in which $\alpha_2$ appears with odd positive 
coefficient. Each of the sets $\allroots_{2, 1}$ and $\allroots_{2, 3}$ forms a distributive lattice under the usual
ordering on roots. The poset $\allroots_{2, 1}$ is isomorphic (via a unique isomorphism) to $J([3] \times [5])$ 
and is in canonical bijection with the $56$ spreads, where the lowest element $\alpha_2$ corresponds to the spread 
$123$, and the highest element corresponds to the spread $678$. The poset $\allroots_{2, 3}$ is a chain which is in 
canonical bijection with the set of
basepoints, where the lowest element corresponds to $1 \in \eight$ and the highest element $\theta$ corresponds
to $8 \in \eight$. Each of the $240$ maximal orthogonal subsets of $\Psi'$ contains seven elements of $\allroots_{2, 1}$
and one element of $\allroots_{2, 3}$, and each such subset corresponds bijectively to a packing of $PG(3, 2)$, as in
Definition \ref{def:slm}.

One advantage of using $\Psi'$ instead of $\Psi$ is that the copy of $S_8$ that acts on $\Psi'$ is the standard parabolic
subgroup of $W(E_8)$ generated by $S \backslash \{s_{\al_2}\}$. On the other hand, the set $\Psi$ has the 
advantages that (a) it has an additive structure, (b) it is compatible with the standard coordinate system of the $E_8$ 
lattice, and (c) it supports natural actions by signed permutations. We will primarily work with $\Psi$ rather than 
$\Psi'$ in this paper.
\end{rmk}

\section{Lehmer codes and the Bruhat order on packings}\label{sec:lehmer}

In Section \ref{sec:lehmer}, we define the Bruhat order on packings and its associated Lehmer code. The key to the 
definition is a new integer valued statistic on packings analogous to the length of a permutation, which we 
call the ``height" of a packing. Applying a transposition to a packing always changes the height by an odd number, 
and we define the resulting packing to be higher (respectively, lower) in the partial order if the height increases 
(respectively, decreases). 

The definitions in Section \ref{sec:lehmer} are as combinatorial as possible and do not involve Lie theory, which has 
the advantage of making the definitions straightforward, but has the disadvantage of making the definitions look 
arbitrary. Later, we will relate the partial order on packings to the aforementioned partial order on $\qp_\Psi$ 
(Theorem \ref{thm:qppack}), and we will give two combinatorial interpretations of the height function in terms of the
root system of type $E_8$ (Remark \ref{rmk:residues}).

Let $S_n$ be the symmetric group on $n$ letters. A set of generators for $S_n$ is $S = \{s_1, s_2, \ldots, s_{n-1}\}$,
where $s_i$ is the simple transposition $(i, i+1)$. If $k$ is the smallest number such that 
$w = s_{i_1} s_{i_2} \cdots s_{i_k}$, then we call the word $s_{i_1} s_{i_2} \cdots s_{i_k} \in S^*$ a {\it reduced
expression} for $w$, and we call $k$ the {\it length}, $\ell(w)$ of $w$. An {\it inversion} of an element $w \in S_n$ 
is an ordered pair $(i, j)$ such that $1 \leq i < j \leq n$ and $w(i) > w(j)$. The number of inversions of $w$ is equal 
to $\ell(w)$.

We abbreviate the two-line notation $$
\left(
\begin{matrix}
1 & 2 & \cdots & n\\
a_1 & a_2 & \cdots & a_n\\
\end{matrix}
\right)
$$ for a permutation $w \in S_n$ to $a_1 a_2 \cdots a_n$. With this notation, the unique {\it longest element} 
$w_0 \in S_n$ of maximal length is $n\  n-1\ \cdots\ 2 \ 1$.

Manivel \cite[\S2.1]{manivel01} defines the code $c(w)$ of a permutation $w \in S_n$ to be the sequence 
$(c_1, c_2, \ldots, c_{n-1})$, where $$
c_i = \left| (i, j) : 1 \leq i < j \leq n : w(i) > w(j) \right| 
$$ is the number of inversions of $w$ to the right of $i$. This induces a bijection between $S_n$ and the set of
integer sequences $(c_1, c_2, \ldots, c_{n-1})$ such that $0 \leq c_i \leq n-i$.
Manivel \cite[Remark 2.1.9]{manivel01} proves that if we have $c(w) = (c_1, c_2, \ldots, c_{n-1})$, then a reduced
expression for $w$ given by $$
w = 
(s_{c_1} s_{c_1 - 1} \cdots s_1)
\cdots 
(s_{c_k + k - 1} s_{c_k + k - 2} \cdots s_k)
\cdots
(s_{c_{n-1} + n - 2} s_{c_{n-1} + n - 3} \cdots s_{n-1})
,$$ where there are $n-1$ factors (some of which may be empty) of lengths $c_1, c_2, \ldots, c_{n-1}$.

The code described above is often called a {\it Lehmer code} after D.H. Lehmer, although the idea appears in much
older work of Laisant \cite{laisant1888}. For the purposes of this paper, we need to modify the definition of Lehmer
code as follows.

\begin{defn}\label{def:lehmer}
We define the {\it dual Lehmer code} $\leh(w)$ of a permutation $w \in S_n$ to be the sequence 
$(\leh_n, \leh_{n-1}, \ldots, \leh_2)$, where $$
\leh_j = \leh_j(w) = \left| (i, j) : 1 \leq i < j \leq n : w(i) > w(j) \right| 
$$ is the number of inversions of $w$ to the left of $j$.
\end{defn}

\begin{rmk}\label{rmk:redexp}
The longest element $w_0$ satisfies $w_0^2 = 1$, $w_0(i)=n+1-i$ and $w_0 s_i w_0 = s_{n-i}$. When combined with
the discussion above, this has the following immediate consequences for an element $w \in S_n$.
\begin{itemize}
\item[(i)]{We have $\leh(w) = c(w_0 w w_0)$, and vice versa.}
\item[(ii)]{The function $\leh$ is a bijection from $S_n$ to the set of integer sequences 
$(\leh_n, \leh_{n-1}, \ldots, \leh_2)$ for which $0 \leq \leh_i < i$.}
\item[(iii)]{A reduced expression for 
the permutation $w \in S_n$ with code $\leh(w) = (\leh_n, \leh_{n-1}, \ldots, \leh_2)$ is $$
w = w_n w_{n-1} \cdots w_2
,$$ where $w_j$ is the (possibly empty) length-$\leh_j$ word $w_j = s_{j - \leh_j} \cdots s_{j-2} s_{j-1}$ 
for $1 < j \leq n$.}
\item[(iv)]{If $c(w) = (c_i)$ and $\leh(w) = (\leh_j)$ then $\sum_i c_i = \sum_j \leh_j = \ell(w)$.}
\end{itemize}
\end{rmk}

From now on, ``the code of $w$" refers to the dual Lehmer code $\leh(w)$ unless otherwise stated.

\begin{exa}\label{exa:code}
If $w = 3741652 \in S_7$, then we have $w_0 w w_0 = 6327415$, $\ell(w) = 12$, $c(w) = (2, 5, 2, 0, 2, 1)$ 
and $\leh(w) = (5, 2, 1, 3, 1, 0)$. The reduced expression of $w$ induced by $c(w)$ is $$
(s_2 s_1)(s_6 s_5 s_4 s_3 s_2)(s_4 s_3)()(s_6 s_5)(s_6)
,$$ and the reduced expression of $w$ induced by $\leh(w)$ is $$
(w_7)(w_6)(w_5)(w_4)(w_3)(w_2) = (s_2 s_3 s_4 s_5 s_6)(s_4 s_5)(s_4)(s_1 s_2 s_3)(s_2)()
.$$
\end{exa}

We are now ready to define the Bruhat order on packings and relate it to the dual Lehmer code $\leh$
of $S_8$. As a byproduct, we obtain a standard set of coset representatives for the transitive action of
$S_8$ on $\qpx$. The Bruhat order has some other remarkable properties, several of which we will discuss
in Section \ref{sec:conc}.

\begin{defn}\label{def:cset}
Define $C$ to be the set of 240 permutations given by $$
C := \{w \in S_8 : \leh_7(w) = \leh_6(w) = \leh_4(w) = 0\}
.$$ Given $w \in C$, we define the {\it reduced code} of $w$ to be the sequence 
$\leh'(w) = (\leh_8(w), \leh_5(w), \leh_3(w), \leh_2(w))$.
\end{defn}

The set $C$ may also be characterized as the set of elements $w \in S_8$ that have no inversions of the form
$(i, j)$ where $1 \leq i < j \in \{4, 6, 7\}$.

Recall that the set $\qpx$ of packings of $PG(3, 2)$ is indexed by equivalence classes of Fano planes with
vertices labelled by distinct elements of $\eight$, and that $S_8$ acts transitively on $\qpx$ by permuting
the vertex labels. If $1 \leq k < 8$, we write $S_k$ for the subgroup of $S_8$ that fixes the subset 
$\{k+1, k+2, \ldots, 8\} \subset \eight$ pointwise.

The packing $x_0 \in \qpx$ is the one that corresponds to the set of triples $$
\{127, 136, 145, 235, 246, 347, 567\}
$$ depicted in Figure \ref{fig:fano1}, where we have written $\{i, j, k\}$ as $ijk$ for brevity. 
(The choice of $x_0$ is not arbitrary, and is determined by the theory of quasiparabolic sets \cite{rains13},
\cite[Proposition 6.5]{gx5}.) If $w \in S_8$ and $x' =  wx_0$,
we call $w x_0$ is a reduced expression for $x'$ if $w$ has minimal length among all elements $w' \in S_8$ such that 
$wx_0 = w'x_0$.

\begin{lemma}\label{lem:x0}
Let $x_0$ be the packing defined above and let $ux_0$ be a reduced expression for $x' \in \qpx$.
\begin{itemize}
\item[{\rm (i)}]{We have $s_5 x_0 = s_1 x_0$, $s_5 s_4 x_0 = s_3 s_4 x_0$, and $s_6 x_0 = s_2 x_0$.}
\item[{\rm (ii)}]{If $w \in \{(12)(34), (13)(24), (14)(23)\}$ and $y \in S_3$, then we have $wy x_0 = yx_0$.}
\item[{\rm (iii)}]{If $u \in S_4$, then there exists $v \in S_3$ such that
$vx_0$ is a reduced expression for $x'$.}
\item[{\rm (iv)}]{If $u \in S_6$, then there exists $v \in S_5$ such that
$v x_0$ is a reduced expression for $x'$.}
\item[{\rm (v)}]{If $u \in S_7$, then there exists $v \in S_5$ such that
$vx_0$ is a reduced expression for $x'$ and $\leh_4(v)=\leh_6(v)=\leh_7(v)=0$.}
\end{itemize}
\end{lemma}

\begin{proof}
Part (i) follows by direct calculation. 

The case $y = 1$ of (ii) also follows by direct calculation. The general case of (ii) follows from the 
case $y = 1$ and the fact that $\{1, (12)(34), (13)(24), (14)(23)\}$ is a normal subgroup of $S_4$.

The hypothesis $u \in S_4$ of (iii) implies that $u_8=u_7=u_6=u_5=1$, so that $u = u_4 u_3 u_2$ and
$u_3 u_2 \in S_3$. If $u_4 = 1$ then (iii) follows by taking $v=u$. Otherwise, we deduce from (ii) that 
$s_3 u_3 u_2 x_0 = s_1 u_3 u_2 x_0$, and we obtain $v$ by replacing the rightmost $s_3$ in $u_4$ by
$s_1$. The hypothesis that $ux_0$ is reduced implies that $vx_0$ is reduced, and we have $v \in S_3$ because
$v$ contains no occurrences of $s_3$. This completes the proof of (iii).

The hypothesis $u \in S_6$ of (iv) implies that $u_8=u_7=1$. By applying (iii) to the reduced expression
$u_4 u_3 u_2 x_0$, we may assume without loss of generality that $u_4=1$. If we have $u_6=1$, then (iv) 
follows by taking $v=u$. If not, we have $u_6 = u' s_5$, where $\ell(u') = \ell(u_6)-1$ and $u' \in S_5$.
There are two cases to consider.

In the first case, we have $u_6 \ne 1$ and $u_5=1$. By (iii), we may assume that $u_4=1$, which implies 
by (i) that $$
s_5 u_5 u_4 u_3 u_2 x_0 = s_5 u_3 u_2 x_0 = u_3 u_2 s_5 x_0 =u_3 u_2 s_1 x_0
.$$ The assertion of (iv) follows by setting $v = u' u_3 u_2 s_1$.

In the second case, we have $u_6 \ne 1$ and $u_5 \ne 1$, so the reduced expression $u_5$ is of the form
$u'' s_4$ where $\ell(u'') = \ell(u_5)-1$ and $u'' \in S_4$. By (iii), we may assume that $u_4=1$, which 
implies by (i) that $$
s_5 u_5 u_4 u_3 u_2 x_0 = s_5 u'' s_4 u_3 u_2 x_0 = u'' u_3 u_2 s_5 s_4 x_0 = u'' u_3 u_2 s_3 s_4 x_0
.$$ The proof of (iv) is completed by setting $v' = u' u'' u_3 u_2 s_3 s_4$.

The hypothesis $u \in S_7$ of (v) implies that $u_8=1$. Suppose first that $u_7 \ne 1$, so that we have 
$u_7 = y s_6$ for some $y \in S_6$ such that $\ell(y) = \ell(u_7)-1$. By applying (iv) to the reduced 
expression $u_6 u_5 u_4 u_3 u_2 x_0$, we may assume that $u_6=1$. It follows from (i) that we have $$
u x_0 = y s_6 u_5 u_4 u_3 u_2 x_0 = y u_5 u_4 u_3 u_2 s_6 x_0 = y u_5 u_4 u_3 u_2 s_2 x_0
.$$ This reduces the proof of (v) to the case $u \in S_6$.

Suppose from now on that $u \in S_6$, so that $u_7=1$ and $u_6 u_5 u_4 u_3 u_2$ is the reduced expression 
for $u$ corresponding to $\leh(u)$. We may assume by (iv) that $u_6=1$, and by applying (iii) to the 
reduced expression $u_4 u_3 u_2 x_0$, we may assume that $u_4=1$, which completes the proof of (v).
\end{proof}

\begin{defn}\label{def:theorder}
If $wx_0$ is a reduced expression for the packing $x \in \qpx$, then we define the {\it height} $\rootht(x)$ 
to be $\ell(w)$. If $x \in \qpx$ is a packing and $t \in S_8$ is a transposition, we write $x \prec tx$ if 
$\rootht(x) < \rootht(tx)$.
We define the {\it Bruhat order} $(\qpx, \leq)$ on packings to be the reflexive, transitive closure of $\prec$.
\end{defn}

The quantum integers $[d]_q$ in the next result are defined as $$
[d]_q = \frac{q^d-1}{q-1} = 1 + q + q^2 + \cdots + q^{d-1}
.$$

\begin{prop}\label{prop:combo}
Let $x_0$ be the packing defined earlier.
\item[{\rm (i)}]{If $wx_0$ is a reduced expression for $x' \in \qpx$, then there exists a reduced expression 
$y x_0$ for $x'$ such that $y \in C$ and $\ell(y) \leq \ell(w)$.}
\item[{\rm (ii)}]{The set $\{cx_0 : c \in C\}$ is a complete, irredundantly described set of reduced expressions 
for the packings of $PG(3, 2)$.}
\item[{\rm (iii)}]{The generating function $\sum_{c \in C} q^{\rootht(cx_0)}$ is 
$[2]_q [3]_q [5]_q [8]_q$.}
\item[{\rm (iv)}]{If $t \in S_8$ is a transposition and $x \in \qpx$ is a packing, then we have 
$\rootht(tx) = \rootht(x) + 1 \mod 2$.}
\item[{\rm (v)}]{The Bruhat order $(\qpx, \leq)$ is a partial order on $\qpx$.}
\end{prop}

\begin{proof}
In the context of (i), let $w_8 w_7 w_6 w_5 w_4 w_3 w_2$ be the factorization of $w$ corresponding to the code 
$\leh(w)$, and define $u = w_7 w_6 w_5 w_4 w_3 w_2 \in S_7$. Because $w x_0$ is reduced, it follows that 
$u x_0$ is reduced. By Lemma \ref{lem:x0} (v), there exists a reduced expression $v x_0$ for $u x_0$ such that 
$\leh_4(v) = \leh_6(v) = \leh_7(v) = 0$, which implies that $v \in C$. Part (i) follows by taking
$y = w_8 v$.

The group $S_8$ acts transitively on the packings, and it follows from (i) that every 
packing $w x_0$ for $w \in S_8$ can be expressed as $c x_0$ for some $c \in C$. We conclude that the list in
the statement of (ii) is complete, and the irredundancy follows from the fact that $|C| = 240$ is equal to the number
of packings.

Suppose that $cx_0$ is not a reduced expression for some $c \in C$. Let $w x_0$ be
a reduced expression for $c x_0$. By (i), there exists a reduced expression $c' x_0$
for $w x_0$ such that $c' \in C$ and $\ell(c') = \ell(w) < \ell(c)$. It follows that
$c$ and $c'$ are distinct elements of $C$ for which $c x_0 = c' x_0$, which is a contradiction and completes
the proof of (ii).

Part (iii) follows from (ii) and the definition of the set $C$.

The stabilizer $H$ of $x_0$ in $S_8$ is isomorphic to the nonabelian simple group $GL(3, 2)$. It follows that
$H$ consists entirely of even elements, or else $H \cap A_8$ would be a normal subgroup of $H$ of index $2$.
If $c \in C$, it follows that the left coset $cH$ consists entirely of even (respectively, odd) elements 
if $\rootht(cx_0)$ is even (respectively, odd). The conclusion of (iv) now follows.

The fact that $\rootht(tx) \ne \rootht(x)$ in the conclusion of (iv) implies that we have $\rootht(x) < \rootht(y)$
whenever $x < y$ in $(\qpx, \leq)$. It follows that the Bruhat order on packings is antisymmetric, which completes
the proof of (v).
\end{proof}

For $d > 0$, we define the chain $[d]$ to be the set $\{0, 1, \ldots, d-1\}$, ordered in the obvious way. 
The product of chains $[d_1] \times [d_2] \times \cdots \times [d_k]$ is the Cartesian product of the
chains $[d_1], [d_2], \ldots, [d_k]$ in which $(a_1, a_2, \ldots, a_k) \leq (b_1, b_2, \ldots, b_k)$ if and
only if we have $a_i \leq b_i$ for all $i$. Such a relation is a covering relation if and only if we
have $b_j = a_j + 1$ for some $j$, and $b_i = a_i$ for all $i \ne j$.

The set $C$ and the function $\leh'$ in the following result were defined in Definition \ref{def:cset}.

\begin{theorem}\label{thm:refines}
The function $L : \qpx \ra [8] \times [5] \times [3] \times [2]$ defined by $L(cx_0) = \leh'(c)$ for $c \in C$
is a bijection whose inverse is a morphism of partially ordered sets between the product order on chains and\
the Bruhat order on $\qpx$. 
\end{theorem}

\begin{proof}
It follows from Proposition \ref{prop:combo} (ii) that $L$ is a bijection.
A morphism of posets is a function $f$ from the poset $(A, \leq_A)$ to $(B, \leq_B)$ with the property that 
$a_1 \leq_A a_2$ implies $f(a_1) \leq_B f(a_2)$. In order to check that $L^{-1}$ is a morphism of posets,
it suffices to verify the property for the covering relations in the product of chains.

If $a < a'$ is a covering relation in $[8] \times [5] \times [3] \times [2]$, then the sequence $a$ can be
obtained from the sequence $a'$ by decreasing one of the entries by $1$ and keeping all other entries
the same. If $cx_0$ and $c'x_0$ are the packings in $\qpx$ corresponding to $a$ and $a'$, and we have
$c, c' \in C$, then the definition of $\leh$ implies that a reduced expression for $c$ can be obtained from a
reduced expression for $c'$ by deleting a single generator. It then follows from 
\cite[Theorem 5.10, Proposition 5.11]{humphreys90} that there exists a reflection (i.e., a transposition) 
$t \in S_8$ such that $tc' = c$, and that $c < c'$ is a covering relation in the Bruhat order on $S_8$.
It now follows from Definition \ref{def:theorder} that $cx_0 < c'x_0$ in the Bruhat order on $\qpx$.
\end{proof}

\begin{rmk}\label{rmk:refines}
\begin{itemize}
\item[(i)]{A property analogous to that of Theorem \ref{thm:refines} is used in \cite[Definition 4.2]{bolognini25} 
to define the notion of a Lehmer code for the Bruhat order of an arbitrary finite Coxeter group.}
\item[(ii)]{The six packings $$
\{x_0, \ s_1 x_0, \ s_2 x_0, \ s_2 s_1 x_0, \ s_1 s_2 x_0, \ s_1 s_2 s_1 x_0\}
$$ form a convex subposet of $(\qpx, \leq)$ that is isomorphic to the strong Bruhat order on $S_3$. The elements
$s_2 s_1 x_0$ and $s_1 s_2 x_0$ have no greatest lower bound, so $(\qpx, \leq)$ is not a lattice, even though 
$[8] \times [5] \times [3] \times [2]$ is a distributive lattice.}
\item[(iii)]{Define the packings $$
x = s_1 x_0 = (12)x_0
= \{127, 135, 146, 236, 245, 347, 567\}
$$ and $$
x' = s_1 s_2 x_0 = (123)x_0
= \{126, 135, 147, 237, 245, 346, 567\}
.$$ The reduced codes of $x$ and $x'$ are $(0, 0, 0, 1)$ and $(0, 0, 2, 0)$, respectively, so we have
$\rootht(x)=1$ and $\rootht(x')=2$. The relation $x < x'$ is a covering relation in $(\qpx, \leq)$, because
the heights differ by $1$ and we have $(13)x=x'$. Despite this, $L(x)$ and $L(x')$ are not comparable in the
order on the product of chains.}
\end{itemize}
\end{rmk}

\section{Signed permutations and quasiparabolic sets}\label{sec:signed}

The Weyl group $W(D_n)$ of type $D_n$ may be thought of as a group of signed permutations of $n$ objects of order 
$2^{n-1}n!$ (see \cite[\S2]{humphreys90}). The objects may be permuted arbitrarily, and the signs of the objects 
may be changed two at a time. The reflections in $W(D_n)$ consist of the transpositions
$\{(i, j) : 1 \leq i < j \leq n\}$ and the signed transpositions $\{\overline{(i, j)} : 1 \leq i < j \leq n\}$.
The signed transposition $t = \overline{(i, j})$ acts by exchanging object $i$ with object $j$ and then changing
the sign of each object, so that $t(i) = \overline{j}$, $t(j) = \overline{i}$, $t(\overline{i}) = j$, and
$t(\overline{j}) = i$. A set $S$ of Weyl group generators is given by the $n$ elements 
$(1, 2), (2, 3), \ldots, (n-1, n)$ together with $\overline{(1, 2)}$. We call an element $w \in W(D_n)$ 
a {\it pure sign change} if we have $w(i) \in \{i, \overline{i}\}$ for all $1 \leq i \leq n$.

Recall that $\Psi = \allroots_{1, 1}$ is the subset of roots of type $E_8$ in which $\al_1$ occurs with coefficient 
$1$. If we denote the reflection in the root $\al_i$ by $s_{\al_i}$, then $s_{\al_i}$ acts on $\Psi$ by place 
permutation as the transposition $(i-2, i-1)$ for $3 \leq i \leq 8$, and $s_{\al_2}$ acts on $\lamp$ by the signed
transposition $\overline{(1, 2)}$. In addition, the reflection in the highest root of type $E_8$, $s_\theta$,
sends elements of $\Psi$ to elements of $\pm \Psi$, but we can regard this as an action on $\Psi$ by identifying
each element $\be \in \Psi$ with the pair $\{\be, -\be\}$. Under this identification, the reflection $s_\theta$ acts
by the signed permutation $\overline{(7, 8)}$.

The action in the previous paragraph may also be regarded as an action on $\lamp$. The bijection 
$t_8 : \lamp \ra \lamm$ then induces an action of these generators on $\lamm$. Here, the generators $s_{\al_i}$ for
$2 \leq i \leq 8$ act by the same signed permutations as before, because they commute with the action of
$t_8$. In contrast, the reflection $s_\theta$ acts as $t_8 s_\theta t_8$, namely the unsigned transposition
$(7, 8)$. 

If we use Theorem \ref{thm:corresp} to describe this action in terms of $AG(6, 2)$ then the permutations
in $S_8$ act as linear transformations, and the sign changes act as translations.

This discussion may be summarized as follows.

\begin{lemma}\label{lem:d8act}
Identify the Weyl group $W(D_8)$ of type $D_8$ with the subgroup of $W(E_8)$ generated by all simple reflections
other than $s_{\al_1}$, together with the reflection $s_\theta$ corresponding to the highest root $\theta$.
\begin{itemize}
\item[{\rm (i)}]{The bijection 
$t_8 \slp : \Psi \ra \lamm$ induces an action of $W(D_8)$ on $\lamm$ via signed place permutations, in which
$s_{\al_i}$ acts as the transposition $(i-2, i-1)$ for $3 \leq i \leq 8$, $s_\theta$ acts as $(7, 8)$, and $s_{\al_2}$ 
acts as $\overline{(1, 2)}$.}
\item[{\rm (ii)}]{If we identify $\lamm \cong \lamp$ with $AG(6, 2)$, then the unsigned permutations of $W(D_8)$
act as linear transformations, and the sign changes act as translations.}
\end{itemize}\qed\end{lemma}  

It follows from Theorem \ref{thm:corresp} that the action of $W(D_8)$ on roots induces an action on $\qpx$. The 
next result shows how to define this action using only the combinatorics of labelled Fano planes, by describing 
the action of a given signed transposition on a given packing in terms of the action of a particular unsigned
transposition on the same packing.

\begin{prop}\label{prop:signact}
Let $x \in \qpx$ be a packing of $PG(3, 2)$ identified with an $\eight$-labelled Fano plane. Let $l \in \eight$ 
be the basepoint of $x$, let $e$ be one of the labels in $x$, and denote the three triples containing
$e$ by $\{e, a_1, b_1\}$, $\{e, a_2, b_2\}$, and $\{e, a_3, b_3\}$. 
\begin{itemize}
\item[{\rm (i)}]{We have $\overline{(l, e)}x = (a_1, b_1)x$ and $\overline{(a_1, b_1)}x = (l, e)x$.}
\item[{\rm (ii)}]{We have $(a_1, b_1)x = (a_2, b_2)x = (a_3, b_3)x$ and
$\overline{(a_1, b_1)}x = \overline{(a_2, b_2)}x = \overline{(a_3, b_3)}x$.}
\item[{\rm (iii)}]{For any $1 \leq i < j \leq 8$, there exist $1 \leq c < d \leq 8$ such that
$\overline{(i, j)}x = (c, d)x$.}
\item[{\rm (iv)}]{For any $1 \leq i < j \leq 8$, there exist $1 \leq c < d \leq 5$ such that either
$(i, j)x = (c, d)x$ or $(i, j)x = \overline{(c, d)}x$.}
\end{itemize}
\end{prop}

\begin{proof}
By symmetry, it suffices to verify (i) and (ii) for a specific choice of $x$, $l$, and $e$. We will consider 
the case $l = 8$, $x = x_0$, and $e = 1$.

We have reduced (i) to the case $\overline{(1, 8)}x_0 = (2, 7)x_0$ and $\overline{(2, 7)}x_0 = (1, 8)x_0$.
The action of $\overline{(1, 8)}$ on $\lamm$ sends $$
\{ v_{127},
v_{136},
v_{145},
v_{235},
v_{246},
v_{347},
v_{567},
v_{8}\} \quad \text{to} \quad 
\{ v_{127},
v_{136},
v_{145},
v_{467},
v_{357},
v_{256},
v_{234},
v_{8}\}
,$$ where we have listed the eight components in the corresponding order. By inspection, the action of $(2, 7)$ on
$x_0$ has the same effect. Similarly, the action of $\overline{(2, 7)}$ on $\lamm$ sends $$
\{ v_{127},
v_{136},
v_{145},
v_{235},
v_{246},
v_{347},
v_{567},
v_{8}\} \quad \text{to} \quad
\{ v_{1},
v_{458},
v_{368},
v_{235},
v_{246},
v_{347},
v_{567},
v_{278}\}
,$$ which by inspection agrees with the action of $(1, 8)$ on $x_0$. This completes the proof of (i).

The reason that $(a_1, b_1)$, $(a_2, b_2)$, and $(a_3, b_3)$ have the same effect on $x$ in (ii) is that they
all have the same effect as the action of $\overline{(l, e)}$ on $x$, by (i). Similarly,
$\overline{(a_1, b_1)}$, $\overline{(a_2, b_2)}$, and $\overline{(a_3, b_3)}$ all have the same effect on $x$
as $(l, e)$, which proves (ii).

If we have $l \in \{i, j\}$ then (iii) follows from the first assertion of (i); if not, then (iii) follows from
the second assertion of (i).

To prove (iv), let us first suppose that $l \in \{i, j\}$. Define $e$ so that $\{l, e\} = \{i, j\}$.
Part (iii) provides three signed transpositions, $\overline{(a_1, b_1)}$, $\overline{(a_2, b_2)}$, and 
$\overline{(a_3, b_3)}$, that have the same effect on $x$ as $(l, e)$. The set $\{l, e, a_1, b_1, a_2, b_2, a_3, b_3\}$
has size $8$, so at least one of the transpositions $(l, e)$, $\overline{(a_1, b_1)}$, $\overline{(a_2, b_2)}$, and 
$\overline{(a_3, b_3)}$ is of the form $(c, d)$ or $\overline{(c, d)}$ such that $\{c, d\} \cap \{6, 7, 8\} = \emptyset$, 
as required.

It remains to deal with the case where $l \not\in \{i, j\}$. Let $e$ be the label for which $\{e, i, j\}$
is a triple of $x$, and let $\{e, a, b\}$ and $\{e, a', b'\}$ be the other triples that contain $e$.
If at least one of the pairs $\{i, j\}$, $\{a, b\}$ and $\{a', b'\}$ is disjoint from $\{6, 7, 8\}$, we define
$\{c, d\}$ to be this pair, and it follows from (iii) that $(c, d)$ and $(i, j)$ have the same effect on $x$,
proving (iv). Otherwise, the pair $\{l, e\}$ is disjoint from $\{6, 7, 8\}$, and we define $\{c, d\} = \{l, e\}$.
It follows from (i) that $\overline{(c, d)}$ and $(i, j)$ have the same effect on $x$, and this completes the proof.
\end{proof}

There is a second way to define the action of $W(D_8)$ on packings, which will be useful later. This is
defined as follows.

Let $R$ be the ring $\Q[x_1, x_2, \ldots, x_8]/I$, where $I$ is the ideal $$
I = \langle x_1^2 - 1, \  x_2^2 - 1, \  \ldots, \ x_8^2 - 1, \ x_1x_2x_3x_4x_5x_6x_7x_8 + 1 \rangle
,$$ so that the image of each $x_i$ in $R$ squares to $1$, and the product of all eight $x_i$ is $-1$.
As a $\Q$-vector space, $R$ is $128$ dimensional, and is the direct sum of five subspaces $R_0, \ldots, R_4$ 
of dimensions $1$, $8$, $28$, $56$, and $35$, where $R_i$ is the image of the span of the 
monomials of degree $i$. The summand $R_4$ is spanned by the $70$ products of four distinct monomials, 
subject to the relations $x_A = -x_B$, where $\{A\ | \ B\} \in P_4$ and we write $x_{\{a_1, a_2, a_3, a_4\}}$ for 
$x_{a_1} x_{a_2} x_{a_3} x_{a_4}$. 

The group $W(D_8)$ acts on the set of $35$ pairs of signed monomials of $R_4$. We will represent each signed 
monomial by its positive expression, so that $x_C$ is represented as $x_C$, and $-x_C$ is represented as $x_D$, 
where $\{C\ | \ D\} \in P_4$.

\begin{defn}\label{def:splitrep}
Let $x$ be an $\eight$-labelled Fano plane with basepoint $l$ and with triples $A_1, A_2, \ldots, A_7$, where
$A_i \in \binom{\eight}{3}$. We define $\mu(x)$ to be the set of $7$ positive monomials in $R_4$ given by $$
\mu(x) = \{x_{A_1 \cup \{l\}}, x_{A_2 \cup \{l\}}, \ldots, x_{A_7 \cup \{l\}}\} 
.$$
\end{defn}

The spreads in the packing $x$ may be recovered from the set $\mu_x$ by ignoring the indeterminate that is 
common to all seven monomials.

\begin{lemma}\label{lem:sevenact}
Let $R$ be the ring defined above.
\begin{itemize}
\item[{\rm (i)}]{The action of the group $W(D_8)$ on the indeterminates $x_1, x_2, \ldots, x_8$ by signed 
permutations induces an action of $W(D_8)$ on $R$ by $\Q$-algebra automorphisms.}
\item[{\rm (ii)}]{The group $W(D_8)$ permutes the set $\{\mu(x) : x \in \qpx\}$, and the induced action 
on $\qpx$ agrees with the action of Proposition \ref{prop:signact}.}
\end{itemize}
\end{lemma}

\begin{proof}
Part (i) holds because the generators of the ideal $I$ are permuted among themselves by signed 
permutations that involve evenly many sign changes.

It follows from the definitions that permutations of $S_8$ act on the set $\{\mu(x) : x \in \qpx\}$ in a 
way that is compatible with the action on $\eight$-labelled Fano planes. The second assertion thus reduces
to verifying that signed transpositions in $W(D_8)$ act according to the rules of Proposition \ref{prop:signact} (i). 
As in the proof of Proposition \ref{prop:signact} (i), the verification can be reduced to the actions of 
$\overline{(1, 8)}$, $(1, 8)$, $\overline{(2, 7)}$, and $(2, 7)$ on $\mu(x_0)$.

We have $$
\mu(x_0) = \{
x_{1278}, \ x_{1368}, \ x_{1458}, \ x_{2358}, \ x_{2468}, \ x_{3478}, \ x_{5678}
\}
.$$ Writing the components in the corresponding order, we have \begin{align*}
\overline{(1, 8)} \mu(x_0) &= \{
x_{1278}, \ x_{1368}, \ x_{1458}, \ -x_{1235}, \ -x_{1246}, \ -x_{1347}, \ -x_{1567} 
\}\\ 
&= \{
x_{1278}, \ x_{1368}, \ x_{1458}, \ x_{4678}, \ x_{3578}, \ x_{2568}, \ x_{2348} 
\} \\
&= (2, 7) \mu(x_0)
,\end{align*} and \begin{align*}
\overline{(2, 7)} \mu(x_0) &= \{
x_{1278}, \ x_{1368}, \ x_{1458}, \ -x_{3578}, \ -x_{4678}, \ -x_{2348}, \ -x_{2568} 
\}\\ 
&= \{
x_{1278}, \ x_{1368}, \ x_{1458}, \ x_{1246}, \ x_{1235}, \ x_{1567}, \ x_{1347} 
\} \\
&= (1, 8) \mu(x_0)
.\end{align*} It follows that the action of $W(D_8)$ on the signed monomials in $R_4$ does induce an action
on the set $\{\mu(x) : x \in \qpx\}$, and that this induced action on packings agrees with the action in
Proposition \ref{prop:signact}. This completes the proof of (ii).
\end{proof}

A key feature of the packings of $PG(3, 2)$ that seems not to have been noticed before is that they form a quasiparabolic
set, in the sense of Rains--Vazirani \cite{rains13}. In \cite[\S6.3]{gx5}, T. Xu and the author proved that the set
$\qp_\Psi$ is a quasiparabolic set for the Weyl group $W(D_7)$. The correspondence of Theorem \ref{thm:corresp} then induces
a partial order on $\qpx$, and we will show (Theorem \ref{thm:qppack}) that this agrees with the partial order that we 
defined combinatorially in Section \ref{sec:lehmer}.

Quasiparabolic sets can be defined for any Coxeter group $W$ with generating set $S$ and set of reflections $T$. However, 
the only examples we study in this paper are associated with the Weyl group $W(D_8)$ and its parabolic subgroups, meaning 
subgroups generated by subsets of $S$. In the case of $W(D_8)$, we have $$
S = \{(i, i+1) : 1 \leq i < 8\} \cup \{\overline{(1, 2)}\} = \{s_2, s_3, \ldots, s_8, s_\theta\}
,$$ and $$
T = \{(i, j) : 1 \leq i < j \leq 8\} \cup 
\{\overline{(i, j)} : 1 \leq i < j \leq 8\} 
.$$ 

\begin{defn}\label{def:qpset}
Let $W$ be a Weyl group with generating set $S$ and set of reflections $T$. A {\it scaled $W$-set} is a pair $(X, \lambda)$, 
where $X$ is a $W$-set and $\lambda : X \rightarrow \Z$ is a function satisfying $|\lambda(sx) - \lambda(x)| \leq 1$ for all $s \in S$.
An element $x \in X$ is {\it $W$-minimal} (respectively, {\it $W$-maximal}) if $\lambda(sx) \geq \lambda(x)$ (respectively, 
$\lambda(sx) \leq \lambda(x)$) for all $s \in S$. We call $\lambda(x)$ the {\it height} of $x$.

A {\it quasiparabolic $W$-set} is a scaled $W$-set $X$ satisfying the following
two properties:
\begin{itemize}
\item[(QP1)]{for any $r \in T$ and $x \in X$, if $\lambda(rx) = \lambda(x)$, then $rx = x$;}
\item[(QP2)]{for any $r \in T$, $x \in X$, and $s \in S$, if $\lambda(rx) > \lambda(x)$ and $\lambda(srx) < \lambda(sx)$, then $rx = sx$.}
\end{itemize}

A quasiparabolic $W$-set is endowed with a partial order $\leq_Q$, called the {\it Bruhat order}. This is the smallest partial order 
such that $x \leq_Q rx$ whenever $x \in X$, $r \in T$, and $\lambda(x) \leq \lambda(rx)$.

We call a subgroup $H \leq W$ a {\it quasiparabolic subgroup} of $W$ if $H$ is the stabilizer of a minimal element in a
transitive quasiparabolic $W$-set.
\end{defn}


The following properties of quasiparabolic sets will be useful in the sequel.

\begin{theorem}[Rains--Vazirani \cite{rains13}]\label{thm:summary}
Let $W$ be a Coxeter group with generating set $S$, and let $X$ be a finite transitive quasiparabolic $W$-set.
\begin{itemize}
\item[{\rm (i)}]{The partially ordered set $(X, \leq_Q)$ has a unique minimal element, $x_0$, and a unique maximal
element, $x_1$.}
\item[{\rm (ii)}]{If $\bfw = s_{i_1} s_{i_2} \cdots s_{i_k}$ is a reduced expression for $w \in W$, then some 
subexpression $\bfu$ of $\bfw$ is a reduced expression for $u \in W$ such that $ux_0$ is a reduced expression
for $wx_0$.}
\item[{\rm (iii)}]{The height, $\lambda(x)$, of any $x \in X$ is equal to the smallest $k \geq 0$ such that we
have $x = w x_0$, where $w \in W$ satisfies $\ell(w) = k$.}
\item[{\rm (iv)}]{The height function $\lambda$ can be reconstructed from knowing the minimal element $x_0 \in X$,
together with its height $\lambda(x_0)$, and its stabilizer $\Stab_W(x_0)$.}
\item[{\rm (v)}]{Suppose that $W_I = \langle I \rangle$ is the standard parabolic subgroup of $W$ generated by
the subset $I \subset S$. If $H$ is a quasiparabolic subgroup of $W_I$, then $H$ is a quasiparabolic subgroup of $W$.}
\end{itemize}
\end{theorem}

\begin{proof}
The five parts of the theorem follow immediately from 
\cite[Corollary 2.10]{rains13}, \cite[Theorem 2.8 (ii)]{rains13}, \cite[Corollary 2.13]{rains13},
\cite[Proposition 2.15]{rains13}, and \cite[Corollary 3.10]{rains13}, respectively.
\end{proof}

\begin{lemma}\label{lem:s8d7}
Identify $S_8$ with the subgroup of $W(D_8)$ generated by $S \backslash \{\overline{(1, 2)}\}$, and identify $W(D_7)$
with the subgroup of $W(D_8)$ generated by $S \backslash \{(7, 8)\}$.
\begin{itemize}
\item[{\rm (i)}]{The action of $S_8$ on $\eight$-labelled Fano planes makes the set $\qpx$ into a transitive 
quasiparabolic  $S_8 \cong W(A_7)$-set with minimal element $x_0$. The height function $\lambda(x)$ is equal to 
$\rootht(x)$, as defined in Definition \ref{def:theorder}.}
\item[{\rm (ii)}]{The action of $W(D_7)$ on $\qpx$ induced by its action the $E_8$ root system makes $\qpx$
into a transitive quasiparabolic $W(D_7)$-set with minimal element $x_0$ and height function $\rootht(x)$.}
\end{itemize}
\end{lemma}

\begin{proof}
Let $H$ be the stabilizer in $S_8$ of $x_0$. Because the basepoint of $x_0$ is $8$, it follows that we have 
$H \leq S_7$, where we identify $S_7$ with the stabilizer of $8 \in \eight$. It follows from 
\cite[\S6.2]{gx5} that $H$ is a quasiparabolic subgroup of $S_7$, and Theorem \ref{thm:summary} (v) then 
implies that $H$ is a quasiparabolic subgroup of $S_8$. By Theorem \ref{thm:summary} (iii), this endows
$\qpx$ with the structure of a transitive quasiparabolic $S_8$-set with height function given by $$
\lambda(y) = \lambda(x_0) + \min_{wx_0 = y} \left\{ \ell(w) \right\}
.$$ We now have $\lambda(x) = \rootht(x)$ by definition of the function $\rootht$ (Definition \ref{def:theorder}),
which proves the assertion of (ii)(a).

By \cite[\S6.3]{gx5}, the set $\qp_\Psi$ is a transitive quasiparabolic $W(D_7)$-set with respect to the action
of $W(D_7) \cong \langle S \backslash \{(7, 8\} \rangle$ on the $E_8$ root system. Theorem \ref{thm:corresp} then
implies that $\qpx$ is a quasiparabolic $W(D_7)$-set with minimal element $x_0$. (The motivation for the 
definitions of $x_0$ and $x_1$ is that they correspond to the minimal and maximal elements of this $W(D_7)$-set.) 
Let $\lambda_D$ denote the height function of this $W(D_7)$-set. It remains to show that we have 
$\lambda_D(x) = \rootht(x)$.

Proposition \ref{prop:signact} (i) implies that we have $\overline{(1, 2)}x_0 = (7, 8)x_0$. Given $x \in \qpx$, we
have a reduced expression $w x_0 = w_8 w_5 w_3 w_2 x_0$ corresponding to the code $\leh$, so that $\rootht(x) = 
\ell(w_8 w_5 w_3 w_2)$. 

We first prove that $\lambda_D(x) \leq \rootht(x)$; in other words, that there is a $W(D_7)$-expression $\bfx$ for 
$x$ of length at most $\rootht(x)$. If $w_8 = 1$ then the expression $\bfx = w_5 w_3 w_2$ satisfies the hypothesis.
If not, then $w_8$ contains the only appearance of the generator $(7, 8)$ in the expression, and we can write
$w_8 = u (7, 8)$, where $\ell(u) = \ell(w_8) - 1$ and $u \in S_8 \cap W(D_7)$. This implies that $$
w_8 w_5 w_3 w_2 x_0 = u (7, 8) w_5 w_3 w_2 x_0 = u w_5 w_3 w_2 (7, 8) x_0 = u w_5 w_3 w_2 \overline{(1, 2)} x_0
,$$ which proves that $\lambda_D(x) \leq \rootht(x)$ because we have 
$\bfx = uw_5w_3w_2 \overline{(1, 2)} \in W(D_7)$.

Suppose for a contradiction that $\bfx$ is not reduced. By Theorem \ref{thm:summary} (ii), we can convert $\bfx$ to 
a reduced expression $\bfy$ by deleting some of the generators. If the single occurrence of $\overline{(1, 2)}$
survives the deletion process, we replace it by an occurrence of $(7, 8)$ and call the resulting expression $\bfz$; 
if not, we define $\bfz = \bfy$. In any case, the element $z \in S_8$ expressed by $\bfz$ gives an $S_8$-expression 
for $zx_0 = yx_0 = x$ that is shorter than the reduced $S_8$-expression $wx_0$, which is a contradiction. We conclude 
that there exists an $S_8$-expression for $x$ has the same length as a $W(D_7)$-reduced expression
for $x$, which proves that $\rootht(x) \leq \lambda_D(x)$ by Theorem \ref{thm:summary} (iii) and completes the proof. 
\end{proof}

Using Lemma \ref{lem:s8d7}, we can make $\qpx$ into a quasiparabolic $W(D_8)$-set by patching together the 
$S_8$ and $W(D_7)$ actions. A key observation is that each element of the generating set $S$ of $W(D_8)$ is
either a generator of $S_8$, or a generator of $W(D_7)$, or both.

\begin{theorem}\label{thm:qppack}
Let $W(D_8)$ be a Weyl group of type $D_8$ with the usual generating set $S$, acting on the set $\qpx$ of packings
of $PG(3, 2)$ as in Proposition \ref{prop:signact} or Lemma \ref{lem:sevenact}, and let $\rootht$ be the
statistic on packings of Definition \ref{def:theorder}.
\begin{itemize}
\item[{\rm (i)}]{The set $\qpx$ is a transitive quasiparabolic $W(D_8)$-set, with minimal element $x_0$ and height 
function $\lambda(x) = \rootht(x)$.}
\item[{\rm (ii)}]{For each of the following parabolic subgroups $W_I \leq W(D_8)$, the set $\qpx$ is a transitive 
quasiparabolic $W_I$-set with minimal element $x_0$ and height function $\lambda(x) = \rootht(x)$, and the 
quasiparabolic order in each case agrees with $(\qpx, \leq)$:
\begin{itemize}
\item[{\rm (a)}]{$S_8 \cong W(A_7)$, generated by $I = S \backslash \{\overline{(1, 2)}\}$;}
\item[{\rm (b)}]{$W(D_5)$, generated by $I = S \backslash \{(5, 6), (6, 7), (7, 8)\}$;}
\item[{\rm (c)}]{any standard parabolic subgroup $W_I$ satisfying $W(D_5) \leq W_I \leq W(D_8)$, including $W(D_6)$,
generated by $I = S \backslash \{(6, 7), (7, 8)\}$, and $W(D_7)$, generated by $I = S \backslash \{(7, 8)\}$.}
\end{itemize}}
\end{itemize}
\end{theorem}

\begin{proof}
Let $S$ be usual set of generators for $W(D_8)$. Choose a generator $s \in S$, a (possibly signed) transposition 
$r \in T$, and a packing $x \in \qpx$. 

Suppose first that $s \in S_8$, meaning that $s$ is an unsigned simple transposition. By Proposition \ref{prop:signact} (i), 
there exists a transposition $r' \in S_8$ such that $rx = r'x$. By Lemma \ref{lem:s8d7} (i), the definition of
scaled $W$-set and conditions (QP1) and (QP2) are all satisfied with respect to $W = S_8$, $s$, $x$, $r'x$
and the function $\rootht$. It follows that the definition is also satisfied with respect to $W = W(D_8)$, $s$, $x$,
$rx$, and the function $\rootht$. On the other hand, if $s \in W(D_7)$, we can apply a similar argument based on
Lemma \ref{lem:s8d7} (ii) and Proposition \ref{prop:signact} (ii), which completes the proof of (i).

Part (ii)(a) follows from Lemma \ref{lem:s8d7} (i), and the statements (b)--(d) of (ii) all follow by restriction 
from (i) except for the claims about transitivity. Proposition \ref{prop:signact} (iv) implies that whenever $t \in S_8$ 
is a transposition and $x \in \qpx$ is a packing, there exists a reflection $r \in W(D_5)$ such that $tx = rx$.
The transitivity in (ii)(b) now follows from the transitivity of (ii)(a), and (ii)(c) is immediate from (ii)(b).
\end{proof}

\begin{rmk}\label{rmk:d5stab}
The group $W(D_5)$ has order $1920$, making it considerably smaller than $A_8$ or $S_8$. The stabilizer of $x_0$ in 
$W(D_5)$ is elementary abelian of order $8$, generated by $(1\ 2)(3\ 4)$, $(1\ 3)(2\ 4)$, and 
$\overline{(1\ 2)}\,\overline{(3\ 4)}$. 
\end{rmk}

\begin{defn}\label{def:oddeven}
Let $x \in \qpx$ be a packing, regarded as an $\eight$-labelled Fano plane. Let $e \in \eight$ be the highest 
label appearing in $x$, and let $$
\{\{e, a_1, b_1\},\ \{e, a_2, b_2\}, \ \{e, a_3, b_3\}\}
,$$ be the three triples of $x$ that contain $e$, ordered so that we have $a_i < b_i$ for all $i$. Let $f$ be the
point of intersection of the line through $\{a_1, a_2\}$ with the line through $\{b_1, b_2\}$. We call the label
$e$ of $x$ {\it odd} (respectively, {\it even}) if $a_1 a_2 a_3$ is a triple and we have $f = a_3$ (respectively,
$b_1 b_2 b_3$ is a triple and we have $f = b_3$). The {\it parity} of a label $e$ is its odd or even status.
\end{defn}

\begin{exa}\label{exa:oddeven}
The labelled Fano plane $$
\{124, \ 235, \ 346, \ 457, \ 156, \ 267, \ 137\}
$$ in Figure \ref{fig:oddeven} is fixed by the $7$-cycle $(1\ 2\ 3\ 4\ 5\ 6\ 7)$, and is sometimes used to illustrate the 
structure constants of the octonions. It has four odd labels, $1$, $3$, $5$, and $7$, witnessed by the triples $$
a_1 a_2 a_3 = 235, \ 124, \ 124, \ \text{and} \ 124
,$$ respectively, and three even labels, $2$, $4$, and $6$, witnessed by the triples $$
b_1 b_2 b_3 = 457, \ 267,\ \text{and}\ 457,
$$ respectively.

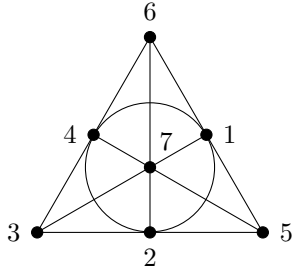
\begin{figure}
\begin{center}
\begin{tikzpicture}[
mydot/.style={
  draw,
  circle,
  fill=black,
  inner sep=1.5pt}
]
\draw
  (0,0) coordinate (A) --
  (3,0) coordinate (B) --
  ($ (A)!.5!(B) ! {sin(60)*2} ! 90:(B) $) coordinate (C) -- cycle;
\coordinate (O) at
  (barycentric cs:A=1,B=1,C=1);
\draw (O) circle [radius=3*1.717/6];
\draw (C) -- ($ (A)!.5!(B) $) coordinate (LC); 
\draw (A) -- ($ (B)!.5!(C) $) coordinate (LA); 
\draw (B) -- ($ (C)!.5!(A) $) coordinate (LB); 
\foreach \Nodo in {A,B,C,O,LC,LA,LB}
  \node[mydot] at (\Nodo) {};    
  \node [left=0.1cm of A] {$3$};
  \node [right=0.1cm of B] {$5$};
  \node [above=0.1cm of C] {$6$};
  \node [right=0.1cm of LA] {$1$};
  \node [left=0.1cm of LB] {$4$};
  \node [below=0.1cm of LC] {$2$};
  \node [above right=0.1cm and 0.001cm of O] {$7$};
\end{tikzpicture}
    \caption{The labelled Fano plane from Example \ref{exa:oddeven}}
\label{fig:oddeven}
\end{center}
\end{figure}
\end{exa}

The next result gives a method to calculate the height of a packing using only the combinatorics of Fano planes,
and proves that the action of the longest element of $S_8$ on the packings gives an antiautomorphism of
the partial order $(\qpx, \leq)$.

\begin{prop}\label{prop:nonrec}
Let $x$ be an $\eight$-labelled Fano plane with basepoint $l$. Let $h(x) = H(x) + (8-l)$, where $H(x)$ is the 
number of odd labels in $x$. Let $z_0= (1\ 8)(2\ 7)(3\ 6)(4\ 5) \in S_8$.
\begin{itemize}
\item[{\rm (i)}]{Every label in $x$ has a well-defined parity.}
\item[{\rm (ii)}]{Every label in the packing $x_0$ is even, and every label in the packing $x_1$ is odd.
We have $h(x_0) = 0$ and $h(x_1) = 14$.}
\item[{\rm (iii)}]{If $s = (i, i+1) \in S_8$ is a simple reflection, then we have $h(sx) = h(x) \pm 1$.}
\item[{\rm (iv)}]{We have $\rootht(x) = h(x)$.}
\item[{\rm (v)}]{Two packings are in the same $A_8$-orbit if and only if their heights are congruent modulo $2$.}
\item[{\rm (vi)}]{For any packing $x \in \qpx$, any reflection $r \in W(D_8)$, and any pure sign change $t \in W(D_8)$,
we have $\rootht(rx) = \rootht(x) + 1 \mod 2$ and $\rootht(tx) = \rootht(x) \mod 2$.}
\item[{\rm (vii)}]{For any packing $x \in \qpx$, we have $\rootht(x) + \rootht(z_0 x) = 14$.}
\item[{\rm (viii)}]{If $x, y \in \qpx$ are packings such that $x \leq y$, then we have $z_0 y \leq z_0 x$.}
\end{itemize}
\end{prop}

\begin{proof}
In the notation of Definition \ref{def:oddeven}, the triple $\{e, a_3, b_3\}$ must intersect the triple containing
$a_1$ and $a_2$ in a label $f$ different from $e$, because there is a unique triple containing each of $\{e, a_1\}$ and 
$\{e, a_2\}$. Similarly, the line through $b_1$ and $b_2$ intersects $\{e, a_3, b_3\}$ in either $a_3$ or $b_3$.
This point of intersection must also be $f$, because the triples containing $\{a_1, a_2\}$ and $\{b_1, b_2\}$ must
intersect each other. The cases $f = a_3$ and $f = b_3$ are exclusive and exhaustive, which proves (i).

The assertions of the first sentence of (ii) can be proved by an exhaustive check. The assertions of the second
sentence follow from these and the definition of $h(x)$.
 
Suppose that we have $l \in \{i, i+1\}$ in the situation of (iii). The action of $s_i$ on $x$ does not change the
relative orders of the labels, so we have $H(sx) = H(x)$, but the basepoints of $x$ and $sx$ differ by $1$,
so we have $h(sx) = h(x) \pm 1$, as required. 

We may suppose from now on that $l \not\in \{i, i+1\}$. Let $e$ be the label of $x$ such that $\{e, i, i+1\}$ is 
a triple of $x$. The action of $s_i$ changes the parity of the label $e$, but it does not change
the basepoint. However, $s_i$ does not change the relative order of any of the other pairs $\{a_i, b_i\}$ appearing in 
Definition \ref{def:oddeven}, so $s_i$ does not change the parity of any other label. It follows that we
have $h(sx) = h(x) \pm 1$, which completes the proof of (iii).

By Theorem \ref{thm:summary} (i), every $y \in \qpx$ is a term in a sequence $$
x_0 = y_0 < y_1 < \ldots < y_{14} = x_1
,$$ where for each pair $(y_i, y_{i+1})$ there is a simple transposition $(j, j+1)$ such that $y_{i+1} = (j, j+1)y_i$
and $\rootht(y_{i+1}) = \rootht(y_i) + 1$. By (ii) and (iii), we must have $h(y_i) = i$, and this proves (iv).

Part (v) follows from (iii), (iv), and the transitivity of the $S_8$-action.

To prove (vi), recall that any reflection $r$ is conjugate in $W(D_8)$ to a simple reflection. It follows that $r$
has odd length, and the first assertion of (vi) follows from (iii) and (iv). The second assertion immediately reduces
to the case where $t$ is a double sign change. Proposition \ref{prop:signact} (i) proves that any double sign change 
$t_a t_b = (a, b) \overline{(a, b)}$ acts on $x$ in the same way as some double transposition $(a, b)(c, d)$. The
proof of (vi) now follows from (iii) and (iv).

The action of $z_0$ reverses the order of $\eight$. It follows from Definition \ref{def:oddeven} that $z_0$ sends
odd labels of $x$ to even labels of $z_0 x$, and vice versa, which implies that $H(z_0 x) + H(x) = 7$.
If $x$ has basepoint $l$, then $z_0 x$ has basepoint $9-l$. Since $(8-l) + (8-(9-l)) = 7$, it follows that
$h(z_0 x) + h(x) = 14$, and (vii) now follows from (iv).

To prove (viii), it suffices to check the statement for covering relations. Suppose that $r$ is a transposition
such that $rx = y$ and $\rootht(x) < \rootht(y)$. Because $z_0$ is an involution, it follows that $r' = z_0 r z_0$ is
also a transposition, and that $r'(z_0 y) = z_0x$. It follows from (v) that $\rootht(z_0 y) < \rootht(z_0 x)$, 
which proves that $z_0 y < z_0 x$, as required.
\end{proof}

\section{Signed labelled Fano planes}\label{sec:d7}

Some aspects of the packings of $PG(3, 2)$, such as coset equivalence and the kernels of the group actions, 
are easier to understand if we replace the action of $S_8$ on $\qpx$ by the action of the Weyl group 
$W(D_7) \cong 2^6 \rtimes S_7$ of type $D_7$. From this point of view, it is more convenient to parametrize the 
packings by $\seven$-labelled Fano planes with signed edges. We introduce and develop this notation in 
Section \ref{sec:d7}, and use it to calculate the kernels of the group actions in Theorem \ref{thm:faithful}.

Recall from Section \ref{sec:signed} that the group $W(D_8)$ acts on the $35$ pairs of signed monomials in the graded
component $R_4$ of the ring $R$. Now, instead of representing each signed monomial with its positive expression,
we will represent each signed monomial by an expression involving $x_8$; for example, we represent
$-x_{1458}$ by $-x_{1458}$, and we represent $-x_{2357}$ by $x_{1468}$. The collection of monomials $\mu(x)$
associated to the packing $x$ in Definition \ref{def:splitrep} can then be represented as a $\seven$-labelled
Fano plane with signed triples. For example, $-x_{1458}$ corresponds to the negative triple $-145$, and
$x_{1468}$ corresponds to the positive triple $+146$.

\begin{exa}\label{exa:d7sign}
Let $x_1$ be the maximal element in the partial order on packings. We have \begin{align*}
\mu(x_1) &= \{
x_{1234}, \ 
x_{1256}, \ 
x_{1278}, \ 
x_{1357}, \ 
x_{1368}, \ 
x_{1458}, \ 
x_{1467}
\} \\
&= \{
-x_{5678}, \ 
-x_{3478}, \ 
x_{1278}, \ 
-x_{2468}, \ 
x_{1368}, \ 
x_{1458}, \ 
-x_{2358}
\}.\end{align*} Figure \ref{fig:d7sign} shows this information in the form of a signed $\seven$-labelled
Fano plane. Positive triples are shown with solid lines, and negative triples with dashed lines.

\begin{figure}
\begin{center}
\begin{tikzpicture}[
mydot/.style={
  draw,
  circle,
  fill=black,
  inner sep=1.5pt}
]
\draw[dashed]
  (0,0) coordinate (A) --
  (3,0) coordinate (B) --
  ($ (A)!.5!(B) ! {sin(60)*2} ! 90:(B) $) coordinate (C) -- cycle;
\coordinate (O) at
  (barycentric cs:A=1,B=1,C=1);
\draw[dashed] (O) circle [radius=3*1.717/6];
\draw (C) -- ($ (A)!.5!(B) $) coordinate (LC); 
\draw (A) -- ($ (B)!.5!(C) $) coordinate (LA); 
\draw (B) -- ($ (C)!.5!(A) $) coordinate (LB); 
\foreach \Nodo in {A,B,C,O,LC,LA,LB}
  \node[mydot] at (\Nodo) {};    
  \node [left=0.1cm of A] {$5$};
  \node [right=0.1cm of B] {$6$};
  \node [above=0.1cm of C] {$2$};
  \node [right=0.1cm of LA] {$4$};
  \node [left=0.1cm of LB] {$3$};
  \node [below=0.1cm of LC] {$7$};
  \node [above right=0.1cm and 0.001cm of O] {$1$};
\end{tikzpicture}
       \caption{The signed $\seven$-labelled Fano plane corresponding to $x_1$}
\label{fig:d7sign}
\end{center}
\end{figure}
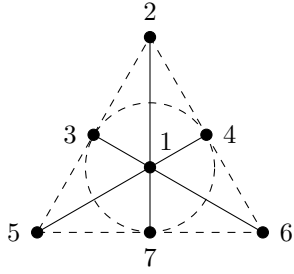
\end{exa}

\begin{defn}\label{def:paschop}
Let $\calf$ be the set of signed $\seven$-labelled Fano planes such that each vertex is incident with an
even number of negative edges. For $1 \leq i \leq 7$ and $F \in \calf$, we define $p_i(F)$ to be the signed 
$\seven$-labelled Fano plane obtained from $F$ by changing the sign of the four edges of $F$ that do not involve 
the label $i$.
\end{defn}

Note that if $F$ is the signed $\seven$-labelled Fano plane corresponding to $x_0$, then $p_1(F)$ is the signed 
$\seven$-labelled Fano plane shown in Figure \ref{fig:d7sign}.

\begin{prop}\label{prop:7label}
Consider the action of the Weyl group $W(D_7)$ on $\qpx$.
\begin{itemize}
\item[{\rm (i)}]{If $F$ is an element of $\calf$, then either $F$ has no negative edges, or we have $F = p_a(F')$,
where $a$ is a label of $F$ and $F' \in \calf$ has no negative edges.}
\item[{\rm (ii)}]{The permutation $w \in S_7 \leq W(D_7)$ acts on signed triples of $F \in \calf$ via $$
w(\pm \{a, b, c\}) = \pm \{ w(a), w(b), w(c) \}
,$$ with no change in sign.}
\item[{\rm (iii)}]{For $1 \leq i < j \leq 7$, we have $$
t_i t_j (\pm \{a, b, c\}) = \begin{cases}
\pm \{a, b, c\} & \text{if $|\{i, j\} \cap \{a, b, c\}|$ is even;}\\
\mp \{a, b, c\} & \text{if $|\{i, j\} \cap \{a, b, c\}|$ is odd.}
\end{cases}
$$}
\item[{\rm (iv)}]{If $\{a, b, c\}$ is a triple of $F$, then $t_a t_b$ negates the four edges not incident with $c$,
and fixes the other three.}
\item[{\rm (v)}]{A signed $\seven$-labelled Fano plane $F$ corresponds to a packing if and only if $F \in \calf$.}
\item[{\rm (vi)}]{Two packings are coset equivalent if and only if the corresponding elements of $\calf$ agree
modulo the signs of their edges.}
\end{itemize}
\end{prop}

\begin{proof}
Suppose $F \in \calf$, and that $F$ has a negative edge. It follows that there is some label $b$ that is incident
with precisely two negative edges. If two particular edges incident with $b$ are designated as negative, then an
exhaustive check proves that there are precisely two ways to assign signs to the remaining four edges so as to maintain
the parity conditions of $\calf$, and both of these assignments are of the form given in (i). Conversely, the set of
four negative edges not incident with a fixed vertex $a$ satisfies the parity conditions of $\calf$, proving (i).

Part (ii) is immediate from the definitions. It also follows from the definitions that the single sign change operator 
$t_i$ negates $\pm \{a, b, c\}$ if $i \in \{a, b, c\}$, and $t_i$ fixes $\pm \{a, b, c\}$ otherwise. Part (iii) follows 
from this, and (iv) follows from (iii).

The unsigned $\seven$-labelled Fano planes lie in $\calf$, because they correspond bijectively to the packings with
basepoint $8$. It follows from (i) and (iv) that the other elements of $\calf$ can be obtained from their unsigned
counterparts by applying a suitable double sign change. We conclude that all elements of $\calf$ correspond to packings.
Part (i) implies that $\calf$ has $240$ elements, because it has precisely eight times as many elements as there are 
unsigned $\seven$-labelled Fano planes. It follows that $\calf$ is in bijective correspondence with $\qpx$, proving
(v).

If $F, F' \in \calf$ agree modulo the signs of their edges, then it follows from (i) and (ii) that there is a pure sign
change $w \in W(D_7)$ such that $w(F) = F'$. Lemma \ref{lem:d8act} (ii) then proves that $F$ and $F'$ are coset equivalent.
Each coset equivalence class contains $2^3 = 8$ elements, and each unsigned $\seven$-labelled Fano plane has seven other
signed versions. This establishes the converse direction of (vi) and completes the proof.
\end{proof}

\begin{rmk}\label{rmk:d7sign}
\begin{itemize}
\item[(i)]{In the cases where $F \in \calf$ has four negative edges, the negative edges and the six labels incident with them
form an example of a {\it Pasch configuration}. This is a collection of six points and four lines, such that each point lies 
on two lines, and each line contains three points.}
\item[(ii)]{The signed $\seven$-labelled Fano planes are reminiscent of the {\it oriented} $\seven$-labelled Fano planes that
are used in the mnemonic for the structure constants of the octonions \cite[\S2.1]{baez02}. However, to the best of our
knowledge, there is no canonical way to turn a signed Fano plane into an oriented Fano plane.}
\end{itemize}
\end{rmk}

The centre of the group $W(D_n)$ is trivial if $n$ is odd, and is equal to $\{1, w_0\}$ when $n$ is even
\cite[\S6.3]{humphreys90}. Here, $w_0$ is the unique element of maximal length, which acts on the root 
system as the scalar $-1$ \cite[\S3.19]{humphreys90}. As a signed permutation, $w_0$ acts as an $n$-fold
sign change. 

The central involution $w_0 \in W(D_8)$ acts on the set $\Psi$ by negation, so it acts trivially on the
packings. We will prove that $w_0$ is the unique nontrivial element in the kernel of the action of
$W(D_8)$ on $\qpx$.

\begin{lemma}\label{lem:signkernel}
Let $N$ be the kernel of the action of $W(D_8)$ on the set of packings $\qpx$. If $w \in N$ acts as a pure
sign change when considered as a signed permutation, then we have $w \in Z(W(D_8)) = \{1, w_0\}$.
\end{lemma}

\begin{proof}
By the preceding remarks, the element $w_0$ acts trivially on $\qpx$, so it remains to deal with the cases
where the Hamming weight of $w$ is $2$, $4$, or $6$.

Suppose that $w$ has Hamming weight $2$ and that it acts as the double sign change 
$t_it_j = (i, j)\overline{(i, j)}$. Let $x \in \qpx$ be a packing with basepoint $l$ containing a triple 
$\{e, i, j\}$. By Proposition \ref{prop:signact} (i), $w$ acts on $x$ as the permutation $(e, l)(i, j)$. It follows
that the basepoint of $wx$ is $e \ne l$, which is a contradiction to the assumption that $w$ is in the kernel.

The element $w$ cannot have Hamming weight $6$, or $ww_0$ would be an element of $N$ with Hamming weight $2$, 
which we have proved cannot happen.

Finally, suppose that $w$ acts as the quadruple sign change $t_at_bt_ct_d$. Let $x$ be a packing containing 
all four labels $\{a, b, c, d\}$ with the property that $\{a, b, c\}$ is a triple of $x$. To prove that 
$wx \ne x$, it suffices (by relabelling if necessary) to consider the case $\{a, b, c, d\} = \{1, 2, 3, 7\}$ 
and $x = x_0$. Direct calculation using Proposition \ref{prop:signact} (i) shows that the triples of the packing $wx_0$ 
are given by $$
wx_0 = (3 7) \overline{(3 7)} (1 2) \overline{(1 2)} x_0
= \{124, 157, 168, 258, 267, 456, 478\}
,$$ which has a basepoint of $3$ and therefore differs from $x_0$, completing the proof.
\end{proof}

\begin{theorem}\label{thm:faithful}
The kernel of the action of $W = W(D_8)$ on the $240$ packings of $PG(3, 2)$ is the centre of $W$, $Z(W) = \{1, w_0\}$.
When restricted to a proper parabolic subgroup of $W$, such as $S_8$, $W(D_7)$, $W(D_6)$, or $W(D_5)$, the 
action is faithful.
\end{theorem}

\begin{proof}
By \cite[\S3.19]{humphreys90}, there is a reduced expression for the longest element $w_0 \in W(D_n)$ that involves
all the generators $S$. It follows (for example, by Matsumoto's Theorem \cite[Theorem 3.3.1]{geck00}) that every
reduced expression for $w_0$ involves all the generators, and therefore that no proper parabolic subgroup of
$W(D_n)$ can contain the longest element $w_0 \in W(D_n)$.

We have $W(D_n) \cong H_n \rtimes S_n$, where the elementary abelian group $H_n$ of order $2^{n-1}$ consists of
pure sign changes, and $S_n$ consists of permutations \cite[\S2]{humphreys90}. Let $\pi_n : W(D_n) \ra S_n$ be 
the surjective homomorphism with kernel $H_n$.

If $n \geq 5$ then the only proper nontrivial normal subgroup of $S_n$ is $A_n$, so any transitive action of
$S_n$ on a set of size bigger than $2$ must be faithful. In particular, $S_7$ acts faithfully on the set of
unsigned $\seven$-labelled Fano planes. 

Let $N_7$ be the kernel of the action of $W(D_7)$ on $\qpx$, and suppose that $w \in N_7$ is a nonidentity
element. Suppose that $w' = \pi_7(w)$ is a nonidentity element of $S_7$. By the previous paragraph, there exists
an unsigned $\seven$-labelled Fano plane $x$ such that $w' x \ne x$. If we now regard $x$ as a signed 
$\seven$-labelled Fano plane with no negative triples, it follows from Proposition \ref{prop:7label} (ii) and
(vi) that $w x \ne x$, because $w' x$, the underlying unsigned Fano plane of $w x$, differs from $x$. This implies 
that $N_7 \leq H_7$. Lemma \ref{lem:signkernel} now implies that $w$ is the longest element of $W(D_8)$, but this 
element does not lie in $W(D_7)$ by the argument of the first paragraph. This proves that $W(D_7)$ acts faithfully 
on $\qpx$.

Now let $N_8$ be the kernel of the action of $W(D_8)$ on $\qpx$. By the previous paragraph, the intersection
$N_8 \cap W(D_7)$ is trivial. It follows that $|W(D_7)|$ divides $|N_8 W(D_7)| = |N_8| \times |W(D_7)|$, which in 
turn divides $|W(D_8)|$. The index of $W(D_7)$ in $W(D_8)$ is $16$, so $N_8$ is a normal subgroup of $W(D_8)$ of
order dividing $16$. It follows that $\pi_8(N_8)$ is a normal $2$-subgroup of $S_8$, but the only such subgroup
of $S_8$ is the trivial group, so we have $N_8 \leq H_8$. By Lemma \ref{lem:signkernel}, we have 
$N_8 \leq Z(W(D_8))$, which proves the first assertion.

The second assertion follows from the first and the fact that no proper parabolic subgroup of $W(D_8)$ can contain
the longest element $w_0$.
\end{proof}

If we delete the label $7$ and the triples incident it from a signed $\seven$-labelled Fano plane, we obtain a signed
$\six$-labelled Pasch configuration. This gives a convenient notation from the point of view of the Weyl group actions of
types $D_6$ and $D_6 + A_1$.

\begin{exa}\label{exa:d6sign}
If we remove the label $7$ and its incident triples from the signed $\seven$-labelled Fano plane in Example \ref{exa:d7sign},
we are left with the signed $\six$-labelled Fano plane shown in Figure \ref{fig:d6sign}, where again we denote positive 
triples with solid lines and negative triples with dashed lines. This diagram is an encoding of the set of monomials $$
 \{
-x_{2468}, \ 
x_{1368}, \ 
x_{1458}, \ 
-x_{2358}
\}.$$ 
\begin{figure}
\begin{center}
\begin{tikzpicture}[
mydot/.style={
  draw,
  circle,
  fill=black,
  inner sep=1.5pt}
]
\draw[dashed]
  (3,0) coordinate (B) --
  ($ (0,0)!.5!(B) ! {sin(60)*2} ! 90:(B) $) coordinate (C) -- (0,0) coordinate (A);
\coordinate (O) at
  (barycentric cs:A=1,B=1,C=1);
\draw (A) -- ($ (B)!.5!(C) $) coordinate (LA); 
\draw (B) -- ($ (C)!.5!(A) $) coordinate (LB); 
\foreach \Nodo in {A,B,C,O,LA,LB}
  \node[mydot] at (\Nodo) {};    
  \node [left=0.1cm of A] {$5$};
  \node [right=0.1cm of B] {$6$};
  \node [above=0.1cm of C] {$2$};
  \node [right=0.1cm of LA] {$4$};
  \node [left=0.1cm of LB] {$3$};
  \node [above right=0.1cm and 0.001cm of O] {$1$};
\end{tikzpicture}
       \caption{Signed $\six$-labelled Pasch configuration corresponding to $x_1$}
\label{fig:d6sign}
\end{center}
\end{figure}
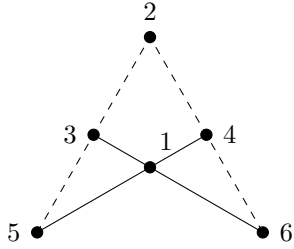
\end{exa}

\begin{rmk}\label{rmk:pasch}
\begin{itemize}
\item[{\rm (i)}]{ The original signed $\seven$-labelled Fano plane can be recovered from the signed $\six$-labelled
Pasch configuration by joining each of the three pairs of non-collinear points to $7$, and assigning signs to the 
triples consistent with the definition of $\calf$.}
\item[{\rm (ii)}]{Every signed $\six$-labelled Fano plane that represents a packing has an even number of negative 
triples. The eight ways to assign signs to the edges subject to this constraint correspond to the eight cosets of 
a maximal totally singular $3$-subspace of $AG(6, 2)$.}
\end{itemize}

\end{rmk}

The Weyl group $W(D_6 + A_1)$ contains three central elements of order $2$, corresponding to the elements
$z_1 = (7, 8)$, $z_2 = \overline{(7, 8)}$, and $z_3 = (7, 8)\overline{(7, 8)}$. The element $z_3$ has the
same effect on a signed $\six$-labelled Pasch configuration as the unique central involution in $W(D_6)$: it
changes the sign of every triple. The element $z_2$ acts on each triple in the Pasch configuration by 
replacing each triple $\{a, b, c\}$ by its complement $\six \backslash \{a, b, c\}$, without changing the
sign; this is a well-known operation on Pasch configurations called a {\it Pasch move}. The element $z_1$
acts replacing each triple by its complement and then changing the sign of the triple; doing this changes 
the height of a packing by $1$, because $(7, 8)$ is a simple transposition.

\begin{exa}\label{exa:paschmove}
Performing the Pasch move $\overline{(7, 8)}$ on the signed $\six$-labelled Fano plane in Example \ref{exa:d6sign} 
results in the monomials $$
 \{
x_{2467}, \ 
-x_{1367}, \ 
-x_{1457}, \ 
x_{2357}
\} = 
 \{
-x_{1358}, \ 
x_{2458}, \ 
x_{2368}, \ 
-x_{1468}
\}
.$$ The resulting signed $\six$-labelled Pasch configuration is shown in Figure \ref{fig:paschmove}. The configurations
in Figures \ref{fig:d6sign} and \ref{fig:paschmove} are dual to each other in the sense that positive (respectively, 
negative) edges in one configuration correspond to triangles with an even (respectively, odd) number of negative sides
in the other configuration.

\begin{figure}
\begin{center}
\begin{tikzpicture}[
mydot/.style={
  circle,
  fill=black,
  inner sep=1.5pt}
]
\draw[dashed]
  (3,0) coordinate (B) --
  ($ (0,0)!.5!(B) ! {sin(60)*2} ! 90:(B) $) coordinate (C) -- (0,0) coordinate (A);
\coordinate (O) at
  (barycentric cs:A=1,B=1,C=1);
\draw (A) -- ($ (B)!.5!(C) $) coordinate (LA); 
\draw (B) -- ($ (C)!.5!(A) $) coordinate (LB); 
\foreach \Nodo in {A,B,C,O,LA,LB}
  \node[mydot] at (\Nodo) {};    
  \node [left=0.1cm of A] {$6$};
  \node [right=0.1cm of B] {$5$};
  \node [above=0.1cm of C] {$1$};
  \node [right=0.1cm of LA] {$3$};
  \node [left=0.1cm of LB] {$4$};
  \node [above right=0.1cm and 0.001cm of O] {$2$};
\end{tikzpicture}
       \caption{The signed $\six$-labelled Pasch configuration of Example \ref{exa:paschmove}}
\label{fig:paschmove}
\end{center}
\end{figure}
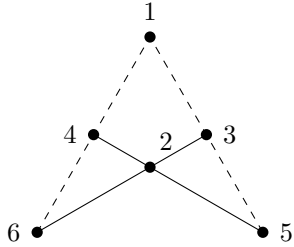
\end{exa}

\begin{rmk}\label{rmk:cubeedge}
The action of $W(D_6)$ on the packings contains, as a substructure, the action of $S_6$ on the $30$
unsigned Pasch configurations. This action is equivalent to the action of $S_6$ on the $30$ rotationally
distinct ways to colour the six opposite pairs edges of a cube with distinct colours, where $S_6$ acts by 
permuting the colours. The six points of the Pasch configuration correspond to the six opposite pairs of
edges, the four triples correspond to the four opposite pairs of vertices, and the three labels of a
triple specify which three colours meet at each vertex.

This action is not equivalent to the action of $S_6$ on the $30$ rotationally distinct face colourings of
a cube by six colours, even up to applying outer automorphisms of $S_6$. Although both actions have a point 
stabilizer that is abstractly isomorphic to $S_4$, the stabilizer of a face colouring contains $4$-cycles,
but the stabilizer of an edge colouring consists entirely of even permutations.

The unsigned Pasch configurations correspond to an interval in $(\qpx, \leq)$ that contains the minimal
element, and this makes the $30$ orbits of edge-coloured cubes into a partially ordered set. The $S_6$-action
extends to an $S_7$-action on the same set, corresponding to the action of $S_7$ on the $30$ unsigned
$\seven$-labelled Fano planes. The signed transposition $\overline{(7, 8)}$ acts on an unsigned Pasch 
configuration as a Pasch move, and this commutes with the action of $S_6$.
\end{rmk}

\section{Concluding remarks}\label{sec:conc}

A great deal is known about the combinatorics of the packings of $PG(3, 2)$. We will not attempt to summarize 
this here; instead, we recommend \cite{hirschfeld85}, \cite{mathon97overlarge}, \cite{mathon97seven}, and 
\cite{polster99} as starting points for further reading. We conclude by describing some properties of the 
packings of $PG(3, 2)$ that may not be familiar to experts on design theory or finite geometry. Most of 
these properties relate to the partial order on the packings.

\begin{rmk}\label{rmk:topology}
The poset $(\qpx, \leq)$ is interesting from the point of view of combinatorial topology. It follows from 
the results of \cite[\S6]{rains13} that the poset $\qpx$ is {\it Eulerian}, which means that every 
nontrivial closed interval in the poset contains the same number of elements of odd rank as even rank.
Furthermore, each interval $[x, y] \subseteq \qpx$ is lexicographically shellable, and if 
$d = \rootht(y) - \rootht(x) = d+2 \geq 2$, then the order complex of $[x, y]$ is homeomorphic to 
a $d$-sphere. The {\it order complex} of $[x, y]$ is the simplicial complex formed by the chains
in the open interval $(x, y)$, where the word ``chain'' has both its combinatorial and its topological
meaning.
\end{rmk}

\begin{rmk}\label{rmk:hecke}
The Weyl groups appearing in this paper have $q$-analogues known as Hecke algebras. Using the partial 
order $(\qpx, \leq)$, one can define actions of the Hecke algebras on the span of the $240$ packings, using
the formula given in \cite[Theorem 7.1]{rains13}.
\end{rmk}

\begin{rmk}\label{rmk:schmidt}
It is not a complete coincidence that the number of packings of $PG(3, 2)$ agrees with the number of
roots of type $E_8$. The {\it orthogonality graph} $\Gamma_1$ of the $E_8$ root system has vertices given
by the $120$ positive roots of type $E_8$, with two vertices being adjacent if they are orthogonal.
Following \cite[\S6.3]{gx5}, we can create a graph $\Gamma_2$ with vertices labelled by the $120$ elements 
in an $A_8$-orbit of $\qp_\Psi$, with adjacency given by disjointness. Although the graphs $\Gamma_1$ and 
$\Gamma_2$ are not isomorphic, Schmidt \cite{schmidt24} proved they are {\it quantum isomorphic} in the 
sense of \cite{atserias19}. The graph $\Gamma_2$, which is denoted by $G_2$ in 
\cite[Table 2.2]{mathon97overlarge}, may also be constructed directly from the packings, by taking
the vertices to be the elements of one $A_8$-orbit of packings, with two vertices being adjacent if they
have no common spread. It follows from Theorem \ref{thm:faithful} that the group $\text{Aut}(G_2)$ of 
order $32 \times 8!$ listed in \cite[Table 2.2]{mathon97overlarge} is $W'/Z(W)$, where $W \cong W(D_8)$.
Both the graphs $\Gamma_1$ and $\Gamma_2$ are strongly regular with parameters $(120, 63, 30, 36)$.
\end{rmk}

\begin{rmk}\label{rmk:residues}
The height function for $(\qpx, \leq)$ has a combinatorial interpretation: it is the number of ``residues"
associated to each element $\calr \in \qp_\Psi$. A {\it residue} of $\calr = \{\be_1, \be_2, \ldots, \be_8\}$
with respect to $\Psi$ is an element $\gamma \in \Psi$ with the property that $s_{\be_i}(\gamma)$ is a positive root 
for each reflection $s_{\be_i}$ corresponding to $\be_i \in \calr$. Residues are introduced and developed 
in \cite{gx6}.

We may also compute the height of a packing corresponding to an $\eight$-labelled Fano plane with basepoint $l$ 
by counting its residues with respect to the set of roots $\Psi' = \allroots_{2, 1} \ \dot\cup \ \allroots_{2, 3}$. It 
can be shown that an element of $\allroots_{2, 3}$ is a residue with respect to $\Psi'$ if and only if it corresponds 
to a basepoint $l' > l$, and that an element of $\allroots_{2, 1}$ is a residue with respect to $\Psi'$ if and only 
if it corresponds to a triple of non-collinear points where each point has a higher label than that of the point 
collinear with the other two. For example, the packing $x_0$ has no residues with respect to $\Psi'$, and the packing 
$x_1$ has $14$ residues with respect to $\Psi'$, given by $$
\{2, 3, 4, 5, 6, 7, 8\} \ \cup \ \{378, 468, 478, 567, 568, 578, 678\}
.$$
\end{rmk}

\begin{rmk}\label{rmk:games}
Because $x_0$ is the minimal element of $(\qpx, \leq)$, it follows that $x_0$ is the only 
element of $\qpx$ that has no residues with respect to $\Psi$. In the context of the $W(D_7)$-action, this information 
can be interpreted as the winning strategy for a two-player game, as follows.

It follows from Example \ref{exa:thcthn} that the packing $x_0$ is identified with the maximally separated 
subset $y_0$ of $\lamp$ given by 
\begin{align*}
\slp(y_0) = \{&
\{1278\ |\ 3456\}, \ 
\{1368\ |\ 2457\}, \ 
\{1458\ |\ 2367\}, \ 
\{2358\ |\ 1467\}, \ \\
& \{2468\ |\ 1357\}, \ 
\{3478\ |\ 1256\}, \ 
\{5678\ |\ 1234\}, \ 
\{12345678\ |\ \}\}
\end{align*}

For each of the eight partitions, we can take the half of the partition that does not contain $8$ and 
interpret it as an arrangement of counters on a horizontal strip with seven positions. Each position can 
hold at most one counter, and the total number of counters is even, so there are $64$ legal game 
positions. A legal move in the game consists of either (a) moving a counter to an empty space somewhere 
strictly to its left, or (b) removing any two counters. The winner is the player who removes the last two counters.

The eight positions shown in Figure \ref{fig:kernel}, which correspond to the eight partitions of $\slp(y_0)$, 
are the {\it kernel} of the game in the sense of \cite[\S51.7.3]{vonneumann07}. The kernel is the set of 
winning positions (assuming optimal play) for the previous player to move. Let $\calr = \{\be_1, \ldots, \be_8\}$
be the orthogonal subset of $\Psi$ that corresponds to $x_0$. If the game position corresponds to an element 
$\gamma \in \Psi \backslash \calr$, then it follows (because $\gamma$ is not a residue of $\calr$) that there exists at 
least one $\be_i \in \calr$ such that $s_{\be_i}(\gamma)$ is negative. The rules then ensure that it is possible to 
move from position $\gamma$ to position $\be_i$ in a single turn. Playing any such move constitutes a winning 
strategy, and playing any other move will allow the opponent to win.

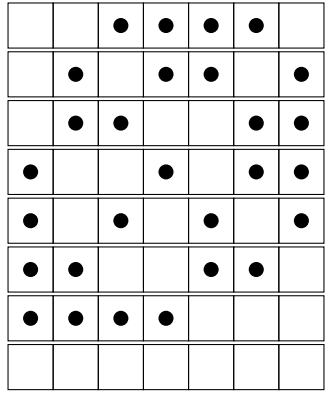
\begin{figure} 
\begin{center}
\begin{tikzpicture}
  \def\s{0.6} 
  \def\dotradius{0.1} 
  \foreach \i in {0,1,...,6} {
    \draw (\i*\s, 0) rectangle ++(\s, \s);
}
\foreach \i in {2,3,4,5} {
      \fill (\i*\s + 0.5*\s, 0.5*\s) circle (\dotradius);
  }
\end{tikzpicture}

\begin{tikzpicture}
  \def\s{0.6} 
  \def\dotradius{0.1} 
  \foreach \i in {0,1,...,6} {
    \draw (\i*\s, 0) rectangle ++(\s, \s);
}
\foreach \i in {1,3,4,6} {
      \fill (\i*\s + 0.5*\s, 0.5*\s) circle (\dotradius);
  }
\end{tikzpicture}

\begin{tikzpicture}
  \def\s{0.6} 
  \def\dotradius{0.1} 
  \foreach \i in {0,1,...,6} {
    \draw (\i*\s, 0) rectangle ++(\s, \s);
}
\foreach \i in {1,2,5,6} {
      \fill (\i*\s + 0.5*\s, 0.5*\s) circle (\dotradius);
  }
\end{tikzpicture}

\begin{tikzpicture}
  \def\s{0.6} 
  \def\dotradius{0.1} 
  \foreach \i in {0,1,...,6} {
    \draw (\i*\s, 0) rectangle ++(\s, \s);
}
\foreach \i in {0,3,5,6} {
      \fill (\i*\s + 0.5*\s, 0.5*\s) circle (\dotradius);
  }
\end{tikzpicture}

\begin{tikzpicture}
  \def\s{0.6} 
  \def\dotradius{0.1} 
  \foreach \i in {0,1,...,6} {
    \draw (\i*\s, 0) rectangle ++(\s, \s);
}
\foreach \i in {0,2,4,6} {
      \fill (\i*\s + 0.5*\s, 0.5*\s) circle (\dotradius);
  }
\end{tikzpicture}

\begin{tikzpicture}
  \def\s{0.6} 
  \def\dotradius{0.1} 
  \foreach \i in {0,1,...,6} {
    \draw (\i*\s, 0) rectangle ++(\s, \s);
}
\foreach \i in {0,1,4,5} {
      \fill (\i*\s + 0.5*\s, 0.5*\s) circle (\dotradius);
  }
\end{tikzpicture}

\begin{tikzpicture}
  \def\s{0.6} 
  \def\dotradius{0.1} 
  \foreach \i in {0,1,...,6} {
    \draw (\i*\s, 0) rectangle ++(\s, \s);
}
\foreach \i in {0,1,2,3} {
      \fill (\i*\s + 0.5*\s, 0.5*\s) circle (\dotradius);
  }
\end{tikzpicture}

\begin{tikzpicture}
  \def\s{0.6} 
  \def\dotradius{0.1} 
  \foreach \i in {0,1,...,6} {
    \draw (\i*\s, 0) rectangle ++(\s, \s);
}
\end{tikzpicture}
       \caption{Winning positions in the game of Remark \ref{rmk:games}}
\label{fig:kernel}
\end{center}
\end{figure}

A more concise way to describe this game is that the game positions are the weights of a $64$-dimensional half-spin 
representation for a simple Lie algebra of type $D_7$, a legal move consists of adding a positive root of type $D_7$
to produce another weight, and the winner is the first player to reach the highest weight.
\end{rmk}



\bibliographystyle{plain}
\bibliography{the240.bib}

\end{document}